\documentclass[11pt, a4paper, reqno]{amsart}
\usepackage{amssymb,latexsym,amsmath,amscd,amsfonts,a4wide,bbm,amsthm}
\usepackage{latexsym}
\usepackage{mathtools}
\usepackage[mathscr]{eucal}
\usepackage{bm}
\usepackage{color}
\usepackage{hyperref}
\usepackage[capitalise]{cleveref}
\usepackage{marginnote}

\hypersetup{
	colorlinks,
	citecolor=teal,
	filecolor=black,
	linkcolor=blue,
	urlcolor=black
}
\usepackage[top=1in,bottom=1.2in,left=1in,right=1in]{geometry}
\newif\ifdraft
\draftfalse
\newsavebox{\savepar}

\usepackage[dvipsnames]{xcolor}

\usepackage{graphicx}
\usepackage{subfig}

\newtheorem{thm}{Theorem}[section]

\newtheorem{lem}[thm]{Lemma}

\newtheorem{prop}[thm]{Proposition}
\newtheorem{defn}[thm]{Definition}
\newtheorem{rem}[thm]{Remark}

\theoremstyle{rem}
\numberwithin{equation}{section}

\newcommand{\pa} {\partial}
\newcommand{\R}{\mathbb{R}}
\newcommand{\N}{\mathbb{N}}
\newcommand{\ds}{\displaystyle}
\renewcommand{\leq}{\leqslant}
\newcommand{\C}{\mathbb{C}}

\def\textmatrix#1&#2\\#3&#4\\{\bigl({#1 \atop #3}\ {#2 \atop #4}\bigr)}
\def\dispmatrix#1&#2\\#3&#4\\{\left({#1 \atop #3}\ {#2 \atop #4}\right)}

\newcommand{\norm}[1]{\left\Vert#1\right\Vert}

\def\hmath$#1${\texorpdfstring{{\rmfamily\textit{#1}}}{#1}}
\newcommand{\rd}{{\rm d}}

\title[]{\sc Some controllability results for Linearized Compressible Navier-Stokes System with Maxwell's Law}

\author{}

\author{Sakil Ahamed}
\address{Sakil Ahamed, Debanjana  Mitra \newline \indent Department of Mathematics, Indian Institute of Technology Bombay, \newline \indent 
	Powai, Mumbai - 400076, India.}
\email{sakil@math.iitb.ac.in, deban@math.iitb.ac.in}
\author[Debanjana Mitra]{Debanjana Mitra$^{*}$}
\thanks{$^{*}$ Corresponding author}
\thanks{Debanjana Mitra acknowledges the support from an INSPIRE faculty fellowship, RD/0118-DSTIN40-001.}

\keywords{Maxwell's system,  Null controllability, Approximate controllability, finite dimensional control, Gaussian beams.}

\subjclass[2010]{35Q30, 93B05, 93B07, 35Q35}

\begin{document}

	\begin{abstract}
	In this article, we study the control aspects of the one-dimensional compressible Navier-Stokes equations with Maxwell's law linearized around a constant steady state with zero velocity. We consider the linearized system with Dirichlet boundary conditions and with interior controls. We prove that the system is not null controllable at any time
using localized controls in density and stress equations and even everywhere control in the velocity equation. However, we show that the system is null controllable at any time if the control acting in density or stress equation is everywhere in the domain. This is the best possible null controllability result obtained for this system. Next, we show that the system is approximately controllable at large time using localized controls. Thus, our results give a complete understanding of the system in this direction. 
		
	\end{abstract}
	
	\maketitle

	%	\maketitle
	
	\section{ Introduction and main results}

Compressible Navier-Stokes system is one of the important models describing fluid flows. We consider the compressible Navier-Stokes equations with Maxwell's law (see, \cite{hu2019global} and references therein). 
The one dimensional compressible Navier-Stokes system with Maxwell's Law is given by :
\begin{equation}\label{eq1}
		\left.
		\begin{aligned}
			&\partial_t\hat{\rho}+\partial_x(\hat{\rho}\hat{u})=0 \qquad\qquad\qquad\qquad\text{ in }(0, T)\times (0, \pi),\\&\partial_t\left( {\hat{\rho}\hat u}\right) +\partial_x\left( {\hat{\rho}\hat u^2}\right) +\partial_xp =\partial_x\hat{S}\qquad\text{in }(0, T)\times (0, \pi),\\&\kappa\partial_t\hat{S}+\hat{S}=\mu\partial_x\hat{u}\qquad\qquad\qquad\qquad\text{in }(0, T)\times (0, \pi),
		\end{aligned}
		\right\}
	\end{equation}
where the pressure $p$ satisfies the following constitutive law :
	\begin{equation}
		p(\rho)=a\rho^{\gamma},\qquad a>0,\: \gamma\geq 1.
	\end{equation}
In the system, $\hat{\rho}$, $\hat{u}$ and  $\hat{S}$ denote the density, the velocity and the stress tensor of the fluid, respectively. Here $\mu$ and $\kappa$ are positive constants where $\mu $ denotes the fluid viscosity and $\kappa$ is the relaxation time describing the time lag in the response of the stress tensor to the velocity gradient. 

Relation $\eqref{eq1}_3$ is first proposed by Maxwell, in order to describe the relation of stress tensor and velocity gradient for a non-simple fluid.

In this paper, we study the controllability properties of the one-dimensional compressible Navier-Stokes system with Maxwell's law and with Dirichlet boundary conditions
linearized around a stationary trajectory of \eqref{eq1} with zero velocity. This is the first step to study the controllability aspects of  \eqref{eq1} around such trajectory. We have the  following observation. 
\begin{rem}
Suppose $\left( \rho_s(x), u_s(x), S_s(x)\right), x\in [0,\pi] $ is a stationary trajectory to the system \eqref{eq1} with $u_s(x)=0$, for all $x\in [0,\pi]$. 
Then from the equations, we get $S_s(x)=0$ and $\rho_s(x)=\rho_s$ (constant), for all $x\in [0,\pi]$. 
Thus, the stationary trajectory to the system with zero velocity can only be in the form $(\rho_s, 0, 0)$ where $\rho_s>0$, a constant.
\end{rem}

In this article, we consider the system linearized around a constant steady state $\left( {\rho_s},0,0\right) $, where ${\rho_s}>0 $, with distributed controls $f_1, f_2$ and $f_3$:
	\begin{equation}\label{eq2}
		\begin{dcases}
			\partial_t\rho+{\rho_s}\partial_xu= \mathbbm{1_{\mathcal{O}_{1}}} f_1,&\text{ in }(0, T)\times (0, \pi),\\
			\partial_tu+a\gamma{\rho_s}^{\gamma-2}\partial_x\rho-\frac{1}{{\rho_s}}\partial_xS=\mathbbm{1_{\mathcal{O}_{2}}} f_2,&\text{ in }(0, T)\times (0, \pi),\\
			\partial_tS+\frac{1}{\kappa}S-\frac{\mu}{\kappa}\partial_xu=\mathbbm{1}_{\mathcal{O}_{3}} f_3,&\text{ in }(0, T)\times (0, \pi),\\
		\end{dcases}
	\end{equation}
	where $\mathbbm{1_{\mathcal{O}_{j}}}$ is  the characteristic function of an open set ${\mathcal{O}}_j \subset (0,\pi),$  $j=1,2,3$. We choose the following initial and boundary conditions for the system \eqref{eq2}:
	\begin{equation}\label{bdd-ini}
		\begin{dcases}
			\rho(0)=\rho_0,\quad u(0)=u_0,\quad S(0)=S_0, & \text{ in } (0,\pi),\\
			u(t,0)=0,\quad u(t,\pi)=0,& \text{ in } (0, T).
		\end{dcases}
	\end{equation}
	
\begin{defn}
	The system \eqref{eq2}-\eqref{bdd-ini} is null controllable in $(L^2\left( 0,\pi\right))^3$ at time $T>0$, if for any initial condition $\left( \rho_0,u_0,S_0\right)^\top\in \left(L^2(0,\pi) \right)^3  $, there exist controls $f_i\in L^2\left(0,T ; L^2(\mathcal{O}_{i}) \right), i=1,2,3, $ such that the corresponding solution $\left(\rho,u,S \right) $ of \eqref{eq2}-\eqref{bdd-ini} satisfies
		\begin{equation}
			\left(\rho,u,S \right)^\top(T,x)=(0, 0, 0)^\top \text{ for all } x\in (0,\pi).
		\end{equation}
\end{defn}
\begin{defn}
	The system \eqref{eq2}-\eqref{bdd-ini} is approximately controllable in the space $(L^2\left( 0,\pi\right))^3$ at time $T>0$, if for any initial condition $\left( \rho_0,u_0,S_0\right)^\top\in \left(L^2(0,\pi) \right)^3  $ and any other $\left( \rho_T,u_T,S_T\right)^\top\in \left(L^2(0,\pi) \right)^3  $ and any $\epsilon>0$, there exist controls $f_i\in L^2\left(0,T ; L^2(\mathcal{O}_{i}) \right), i=1,2,3, $ such that the corresponding solution $\left(\rho,u,S \right) $ of \eqref{eq2}-\eqref{bdd-ini} satisfies
	\begin{equation}
	\Big\|(\rho, u, S)^\top(T, \cdot)-(\rho_T, u_T, S_T)^\top \Big\|_{(L^2(0,\pi))^3} < \epsilon. 
	\end{equation}
\end{defn}

Before stating our main results, we mention some related results in this direction from the literature. The existence of the solution of the compressible Navier-Stokes system with Maxwell's law along with the bolw-up results has been studied, for example, in \cite{HuRacke, hu2019global, WangHu} and references therein. If $\kappa = 0$, then the Maxwell's law \eqref{eq1}$_3$ turns into the Newtonion law $\hat{S}=\mu\partial_x\hat{u}$ and the equation \eqref{eq1} becomes Navier-Stokes system of a viscous, compressible, isothermal barotropic fluid (density is function of pressure only), in a bounded domain $(0,\pi)$. The controllability of the compressible Navier-Stokes system around a trajectory has been studied extensively. The compressible Navier-Stokes system linearized around constant steady state $(\rho_s, 0)$ yields a system coupled between an ODE and a parabolic equation. This system with Dirichlet boundary is considered in \cite{chowdhury2012controllability}. There, the authors have proved that the linearized system is not null controllable  using a localized $L^2$-control acting in the velocity equation. However, the system is approximately controllable at any time $T>0$ using a velocity control. The system linearized around $(\rho_s, u_s)$, where $u_s\neq 0$, with periodic boundary conditions is considered in \cite{CM15, chowdhury2014null}. In this case, the system is coupled between a transport and a parabolic equation and the null controllability result using a velocity control is obtained for large enough time $T$. In \cite{Maity15}, the lack of null controllability of the compressible Navier-Stokes system  linearized around $(\rho_s, u_s)$ has been studied in detail. In the case $u_s=0$, the system is not null controllable using any localized $L^2$-controls at any time $T>0$. In the case $u_s\neq 0$, the same result holds for small time. 
The local null controllability of the compressible Navier-Stokes system around a trajectory with non-zero velocity at large time $T>0$ has been obtained in \cite{ervedoza1, ervedoza2, ervedoza2018local, MR-17, MRR-17, MN-19}. 

Another system relevant to our case is the linearized Maxwell system. In \cite{renardy2005viscoelastic}, the control aspects of one-dimensional shear flows of linear Maxwell and Jeffreys system with single or several relaxation modes have been studied. It is worth mentioning that our lineraized system \eqref{eq2} is similar to the viscoelastic shear flow associated to the multimode Maxwell system considered in \cite{renardy2005viscoelastic} in a sense that both cases the system is coupled between ODEs and the spectrum of the linear operators behaves similarly. However, there are slight differences in the coefficients of the equations appearing in \cite{renardy2005viscoelastic} compared to our case.  
In \cite{renardy2005viscoelastic}, the author has proved that the single-mode Maxwell system is exactly cotrollable using a localized velocity control at sufficiently large time $T$, whereas the multimode Maxwell system is approximately controllable  using a localized velocity control at sufficiently large time $T$. We note that both of the cases the hyperbolic behavior of the spectrum of the linear operator is imposing the geometric control conditions and hence the controllability using localized control can be obtained at large time $T$. 
The controllability of viscoelastic flows in higher dimension has also been available, for instance, in \cite{CARA2, chowdhury2017approximate, MMRvisco} and references therein. 

We also mention that the coupled system may arise to model wave equation with memory terms. After doing a suitable change of variables for the memory term, the wave equation with memory term can be written as coupled hyperbolic-ODE system. The controllability of coupled equations has been studied for example, in \cite{zuazuavisco, biccari2019null, ahamed2022lack} and references therein. In \cite{biccari2019null}, the wave equation with memory  on a one-dimensional torus is considered.
Here the authors have proved that the equation is null controllable in a regular space using a moving control at a sufficiently large time. Also they have shown that the system is not null controllable at any time $T>0$ using a localized control.

Now we state our main results obtained in this article. Our first main result is the lack of null controllability of \eqref{eq2}-\eqref{bdd-ini} for initial data in $(L^{2}(0,\pi))^3$ :
	\begin{thm}\label{thm:lack-1-bis}
		Let 
		\begin{equation*}
			\mathcal{O}_{1} \subset (0,\pi), \quad \mathcal{O}_{2} \subseteq  (0,\pi) \quad \mathcal{O}_{3} \subset  (0,\pi),
		\end{equation*}
		be such that $(0,\pi)\backslash \overline{\mathcal{O}_{1}\cup \mathcal{O}_{3}}$ is a nonempty open subset of $(0,\pi)$. Then the system \eqref{eq2}-\eqref{bdd-ini} is not null controllable in $(L^{2}(0,\pi))^3$ at any time $T>0,$ by interior controls  $f_1\in L^2(0,T; L^2(0,\pi))$ supported in  $\mathcal{O}_1$, $f_2\in L^2(0,T; L^2(0,\pi))$ supported in  $\mathcal{O}_2$ and $f_3\in L^2(0,T; L^2(0,\pi))$ supported in $\mathcal{O}_3$.
	\end{thm}

Now we specially consider the cases when $f_1=0$ and $f_3=0$ in \eqref{eq2}-\eqref{bdd-ini} and only the control in the velocity equation is active. 
Observe that, if $f_3=0$, from the stress equation of \eqref{eq2},
	using the boundary condition \eqref{bdd-ini}, we deduce 
	\begin{equation*}
		\frac{d}{dt}\left( \int_{0}^{\pi}S(t,x)e^{\frac{t}{\kappa}}\:\rd x\right)=0,
	\end{equation*}
and therefore, 
	\begin{equation*}
		\int_{0}^{\pi}S(T,x)e^{\frac{T}{\kappa}}\:\rd x=\int_{0}^{\pi}S_0(x)\:\rd x.
	\end{equation*}
	Thus if the system \eqref{eq2}-\eqref{bdd-ini} with $f_3=0$  is null controllable in time $T>0$ then necessarily
	\begin{equation}\label{stressavg}
		\int_{0}^{\pi}S_0(x)\:\rd x=0.
	\end{equation}
	Similarly, if the system \eqref{eq2}-\eqref{bdd-ini} with $f_1=0$  is null controllable in time $T>0$ then necessarily
	\begin{equation}\label{denseavg}
		\int_{0}^{\pi}\rho_0(x)\:\rd x=0.
	\end{equation}

	\noindent
	Now we introduce the space
	\begin{equation}
		L_m^2(0, \pi)=\left\lbrace f\in L^2(0, \pi) : \: \int_{0}^{\pi}f(x)\rd x=0\right\rbrace.
	\end{equation}	
In view of \eqref{denseavg} and \eqref{stressavg}, it is clear that the possible space for \eqref{eq2}-\eqref{bdd-ini} with $f_1=0$ and $f_3=0$ to be null controllable using control acting only the velocity equation is $L_m^2(0, \pi)\times L^2(0,\pi)\times L_m^2(0, \pi)$ instead of $(L^2(0,\pi))^3$. With this observation, in the case when $f_1=0=f_3$, Theorem \ref{thm:lack-1-bis} can be reformulated as below.
\begin{thm}\label{thm:lack-1-bis-velo}
	Let $f_1=0=f_3$ in \eqref{eq2}-\eqref{bdd-ini}. Let $\mathcal{O}_{2} \subseteq  (0,\pi)$. Then the system \eqref{eq2}-\eqref{bdd-ini} is not null controllable in 
	$L_m^{2}(0,\pi) \times L^{2}(0,\pi) \times L_m^{2}(0,\pi),$ at any time $T>0,$ by an interior control $f_2\in L^2(0,T; L^2(0,\pi))$ supported in  $\mathcal{O}_2$ acting only in the velocity equation. 
	\end{thm}

 Our next main result regarding interior null controllability is the following :
	\begin{thm}\label{thm_pos}
		Let $f_2=0=f_3$ in \eqref{eq2}. Then for any $T>0$ and for any $\left( \rho_0, u_0, S_0\right)\in L^{2}(0,\pi) \times L^{2}(0,\pi) \times L_m^{2}(0,\pi)$, there exists a control $f_1\in L^2\left(0,T ; L^2(0,\pi) \right) $ acting everywhere in the density, such that $(\rho,u, S)$, the corresponding solution to \eqref{eq2}-\eqref{bdd-ini}, belongs
		to $C\left( [0,T];  L^{2}(0,\pi) \times L^{2}(0,\pi) \times L_m^{2}(0,\pi)\right) $ and satisfies
		\begin{equation}
			\rho(T,x)= u(T,x)=S(T,x)=0\text{ for all }x\in (0,\pi).
		\end{equation}
	\end{thm}
	\begin{rem}\label{rem_pos}
		The above theorem is also true when there is only one control in the stress equation acting everywhere in the domain, i.e. , for $f_1=0$, $f_2=0$ and $\left( \rho_0, u_0, S_0\right)\in L_m^{2}(0,\pi) \times L^{2}(0,\pi) \times L^{2}(0,\pi)$, then for any $T>0$, there exists a control $f_3\in L^2\left(0,T ; L^2(0,\pi) \right) $ acting everywhere in the stress, such that $(\rho,u, S)$, the corresponding solution to \eqref{eq2}-\eqref{bdd-ini}, belongs
		to $C\left( [0,T];  L_m^{2}(0,\pi) \times L^{2}(0,\pi) \times L^{2}(0,\pi)\right) $ and satisfies
		\begin{equation}
			\rho(T,x)= u(T,x)=S(T,x)=0\text{ for all }x\in (0,\pi).
		\end{equation}
	\end{rem}
	
In view of \cref{thm:lack-1-bis} and \cref{thm:lack-1-bis-velo}, the results in \cref{thm_pos} and \cref{rem_pos} are the best possible controllability results in $(L^2(0,\pi))^3$ using $L^2$-controls. Note that, in \cref{thm_pos} and \cref{rem_pos}, to obtain the controllability results, one of the controls  in the density or stress equation has to be acting everywhere in the domain. The results in \cref{thm:lack-1-bis} and \cref{thm:lack-1-bis-velo} give that any $L^2$-control in the velocity equation even acting everywhere in the domain fails to give null controllability result in $(L^2(0,\pi))^3$. Next question we ask what controllability results for this system can be expected in $(L^2(0,\pi))^3$ using $L^2$-localized controls. It leads us to investigate the approximate controllability of the system and the result is as follows: 
\begin{thm}\label{main_approximate}
		Let $f_2=0$, $f_3=0$ in \eqref{eq2} and $\mathcal{O}_1\subset (0,\pi)$. Let $T >\frac{4\pi}{\sqrt{a\gamma\rho_s^{\gamma-1}+\frac{\mu}{\kappa\rho_s}}} $. The control system \eqref{eq2}-\eqref{bdd-ini} is approximately controllable in $L^{2}(0,\pi) \times L^{2}(0,\pi) \times L_m^{2}(0,\pi)$ at time $T$, by a localized interior
		control $f_1\in L^2\left( 0,T;L^2(\mathcal{O}_1) \right) $ for the density.
	\end{thm}
\begin{rem}\label{remappro}
If we choose $\mathcal{O}_1=(0,\pi)$, i.e., if the control acts everywhere in the density equation, then the system \eqref{eq2}-\eqref{bdd-ini} is approximately controllable in $L^{2}(0,\pi) \times L^{2}(0,\pi) \times L_m^{2}(0,\pi)$ at any time $T>0$.
\end{rem}
	
\begin{rem}\label{remappro1} In fact, if the control is used one of the equations, the approximate controllability of \eqref{eq2}-\eqref{bdd-ini} can be obtained. We have the following results:
\begin{enumerate}
\item Let $f_1=0$, $f_3=0$ in \eqref{eq2} and $\mathcal{O}_2\subset (0,\pi)$. Then the control system \eqref{eq2}-\eqref{bdd-ini} is approximately controllable in $L_m^{2}(0,\pi) \times L^{2}(0,\pi) \times L_m^{2}(0,\pi)$ at time $T$, by a localized interior
		control $f_2\in L^2\left( 0,T;L^2(\mathcal{O}_2) \right) $ for the velocity, if $T >\frac{4\pi}{\sqrt{a\gamma\rho_s^{\gamma-1}+\frac{\mu}{\kappa\rho_s}}} $.
If $\mathcal{O}_2=(0,\pi)$, then the system \eqref{eq2}-\eqref{bdd-ini} is approximately controllable in $L_m^{2}(0,\pi) \times L^{2}(0,\pi) \times L_m^{2}(0,\pi)$ at any time $T>0$.

\vspace{2.mm}

\item Let $f_1=0$, $f_2=0$ in \eqref{eq2} and $\mathcal{O}_3\subset (0,\pi)$. Then the control system \eqref{eq2}-\eqref{bdd-ini} is approximately controllable in $L_m^{2}(0,\pi) \times L^{2}(0,\pi) \times L^{2}(0,\pi)$ at time $T$, by a localized interior control $f_3\in L^2\left( 0,T;L^2(\mathcal{O}_3) \right) $ for the stress, if $T >\frac{4\pi}{\sqrt{a\gamma\rho_s^{\gamma-1}+\frac{\mu}{\kappa\rho_s}}} $.
If $\mathcal{O}_3=(0,\pi)$, then the system \eqref{eq2}-\eqref{bdd-ini} is approximately controllable in $L_m^{2}(0,\pi) \times L^{2}(0,\pi) \times L^{2}(0,\pi)$ at any time $T>0$.
\end{enumerate}
\end{rem}

The proof of the lack of null controllability results (\cref{thm:lack-1-bis} and \cref{thm:lack-1-bis-velo}) is based on duality arguments. It is well-known that the null controllability of a linear system is equivalent to an observability inequality satisfied by the solution of the adjoint problem \cite[Chapter 2]{coron2007control}. The main idea to prove the lack of null controllability results is to construct special solutions of the adjoint problems for which the observability inequality fails. In particular, the highly localized solutions known as `Gaussian beam' serve the purpose. This type of solution is highly concentrated on the characteristic of the PDEs through space-time. Exploiting the fact that the Gaussian beam is concentrated around the characteristic ray and hence the estimate of the observation outside any small neighborhood of the characteristic ray is negligible, the Gaussian beam can be constructed violating the observability inequality in the case when the observation set does not hit all characteristic rays. This is the key idea of the proof of \cref{thm:lack-1-bis} and \cref{thm:lack-1-bis-velo}. The construction of such solutions for the strictly hyperbolic PDEs has been studied in \cite{Ralston}. For the wave equation, the construction of such solutions has been discussed in \cite{MZ02} to study the lack of null controllability of the equation. For the coupled transport-parabolic system, in \cite{ahamed2022lack}, the construction of Gaussian beams and its application to prove the lack of null controllability of the coupled system have been studied. In this paper, we adapt the construction of Gaussian beam given in \cite{ahamed2022lack} to our system \eqref{eq2} coupling of three ODEs (See \cref{secGB}).

The proof of the rest of the results of this article is accomplished using the spectral analysis of $\mathcal{A}$, the linear operator associated to the system \eqref{eq2}-\eqref{bdd-ini}. The spectrum of the linear operator $\mathcal{A}$ consists of a sequence of eigenvalues converging to a real number $\omega_0<0$ and a pair of complex eigenvalues whose real part is converging to a finite number and imaginary part is behaving as $n$ for $|n|\rightarrow \infty$ (see, \cref{secspec}). The presence of this pair of complex eigenvalues with converging real part and diverging imaginary part indicates the hyperbolic behavior of the system. Hence the geometric control condition is expected to appear in studying the controllability of the system using localized control (see, \cref{main_approximate}). However, the presence of the accumulation point in the spectrum of $\mathcal{A}$ may appear as a constraint to have the exact controllability of the system as obtained in \cref{thm:lack-1-bis}. Moreover, the eigenfunctions (or the generalized eigenfunctions) of $\mathcal{A}$ form a Riesz basis on $(L^2(0,\pi))^3$. A similar result holds for the adjoint operator (See, \cref{secRiesz}). Using this, the series representation of the solution of the \eqref{eq2}-\eqref{bdd-ini} and its adjoint problem can be obtained and they are used explicitly to prove \cref{thm_pos} and \cref{main_approximate}. 

The proof of the null controllability result (\cref{thm_pos}) is obtained by a direct method constructing the control explicitly to bring the solution of the system at rest at any finite given time $T$. Using the fact that the eigenfunctions (or the generalized eigenfunctions) of $\mathcal{A}$ form a Riesz Basis, the system \eqref{eq2}-\eqref{bdd-ini} with control acting everywhere in the domain can be projected onto each finite dimensional eigenspaces for each $n\in \mathbb{N}\cup\{0\}$. For any given time $T>0$, 
the null controllability of this finite dimensional system can be obtained using `Hautus lemma'. Moreover, the construction of the control using the finite dimensional controllability operator can be accomplished along with a uniform estimate independent of $n$ (see, \cref{secnull}). Next, summing up these finite dimensional controls obtained for each projected system, we construct a control for the whole system and we show that the control brings the solution of the whole system to rest at time $T>0$. This technique closely follows the technique used in \cite{chowdhury2012controllability} with a suitable modification to our case when the control acts in the density equation instead of the velocity equation as in \cite{chowdhury2012controllability}.

The proof of the approximate controllability result (\cref{main_approximate}) relies on an unique continuation property of the solutions of the adjoint problem (\cite[Chapter 2]{coron2007control}). Using the representation formula of the solution of the adjoint problem, to prove the unique continuation result, it is needed to show the existence of the unique analytic continuation of a certain exponential series. Then using Laplace transform and Cauchy's integral formula, the require result can be concluded. To prove the existence of the unique analytic continuation of the exponential series, the series is splitted into high and low frequency part corresponding to the complex eigenvalues and then an Ingham inequality is used. For the sake of completeness, the proof of the adaptation of the Ingham inequality in the case of complex modes with converging real part is given in the appendix (See, \cref{secappen}). The main idea of this technique is similar as done in \cite{renardy2005viscoelastic} after a slight modification to our case when the density control is considered instead of the velocity control as in \cite{renardy2005viscoelastic}. Also we note that the system considered in \cite[Theorem 5]{renardy2005viscoelastic} is coupled between ODEs with positive coefficients, whereas, in \eqref{eq2}, the coefficient of the lower order term in the density equation is zero and the imaginary part of the eigenvalues of the linear operator is slightly different from that of associated to \cite[Theorem 5]{renardy2005viscoelastic}.

The novelty of our work in this article is that we thoroughly study the controllability aspects of the compressible Navier-Stokes system with Maxwell's law linearized around a zero velocity in $(L^2(0,\pi))^3$ with Dirichlet boundary condition using $L^2$-controls. We give an explicit construction of the Gaussian beam for this coupled system violating the observability inequality. \cref{thm:lack-1-bis} and \cref{thm:lack-1-bis-velo} conclude that the system cannot be null controllable at any time $T>0$ using $L^2$ localized controls in density and stress equations and $L^2$-control in velocity equation even acting everywhere in the domain. Next, using $L^2$-everywhere control in density or stress equation, the null controllability of the system in $(L^2(0,\pi))^3$ at any time $T>0$ is obtained. It turns out that this is the best possible result on the null controllability of the coupled system in $(L^2(0,\pi))^3$. However, we show that the system is approximately controllable in $(L^2(0,\pi))^3$ at sufficiently large time $T$ using a localized $L^2$-control acting at least one of the equations. We also note that, taking $\kappa=0$, all results obtained here can be recovered for compressible Navier-Stokes equations linearzied around $(\rho_s, 0)$. To the best of our knowledge, there is no controllability result available for the compressible Navier-Stokes system with Maxwell's law. Using this result, the null controllability of the nonlinear system around $(\rho_s, 0, 0)$ using everywhere control only in density or stress equation can be studied. Furthermore, as obtained in the spectral analysis of linear operator associated to \eqref{eq2}-\eqref{bdd-ini}, we anticipate that using $L^2$-localized controls, the nonlinear system can be stabilizable with any exponential decay $-\omega$, where $\omega\in (0,\omega_1)$ and this $\omega_1$ can be determined by the accumulation point appearing in the spectrum of the operator. We plan to investigate such questions in our future work.

%%%%%%%%%%%%%%%%%%%%%%%%%%%%%%%%%%%%%%%%%%%%%%%%%%%%%%%%%%%%%%%%%%%%%%%%%%%%%%%%%%%%%%%%%%%%%%%%%%%%%%%%%%%%%%%%%%%%%%%%%%%%%%%%%%%%%%%%%%%%%%%%%%%%%%%%%%%%%%%%%%%%%%%%%%%%%%
	The plan of the paper is as follows. In Section \ref{well-posedness}, we study the well-posedness of the coupled
	system using semigroup theory, and we determine the adjoint of the linear operator associated to the coupled system. Section \ref{Lack of null controllability}
	is devoted to the constructions of Gaussian beam associated to the coupled system \eqref{eq2} and the proof of \cref{thm:lack-1-bis} and \cref{thm:lack-1-bis-velo}. 
	In Section \ref{secspec}, the spectral analysis for the linear operator associated to the coupled system is given. In \cref{secRiesz} it is shown that the eigenfunctions of the linear operator form a Riesz basis in $(L^2(0, \pi))^3$. In \cref{secproj}, the projection of the operators onto eigenspaces is discussed. \cref{secnull} and \cref{secapprox} are devoted to the proof of \cref{thm_pos} and \cref{main_approximate}, respectively. In \cref{secmultiev}, the main modifications of the proofs of theorems when the linear operator has multiple eigenvalues are indicated. In \cref{secappen}, we add the proof of the Ingham inequality adapted to our case.

	%%%%%%%%%%%%%%%%%%%%%%%%%%%%%%%%%%%%%%%%%%%%%%%%%%%%%%%%%%%%%%%%%%%%%%%%%%%%%%%%%%%%%%%%%%%%%%%%%%%%%%%%%%%%%%%%%
	
	%%%%%%%%%%%%%%%%%%%%%%%%%%%%%%%%%%%%%%%%%%%%%%%%%%%%%%%%%%%%%%%%%%%%%%%%%%%%%%%%%%%%%%%%%%%%%%%%%%%%%%%%%

	\section{ Linearized operator}\label{well-posedness}

	Let us define $\mathcal Z=L^2(0, \pi)\times L^2(0, \pi)\times L^2(0, \pi)$ and the positive constant
	\begin{equation}\label{equ-constant}
		b=a\gamma{\rho_s}^{\gamma-2}.
	\end{equation} 
	Let $\mathcal{Z}$ be endowed with the inner product
	\begin{equation}\label{innerproduct}
		\left\langle \begin{pmatrix}
			\rho\\ u\\S
		\end{pmatrix},\begin{pmatrix}
			\sigma\\ v\\\tilde{S}
		\end{pmatrix} \right\rangle_{\mathcal Z}=b\int_{0}^{\pi}\rho\bar{\sigma} dx+\rho_s\int_{0}^{\pi}u\bar{v} dx+\frac{\kappa}{\mu}\int_{0}^{\pi}S\bar{\tilde{S}} dx,
	\end{equation}
where $\bar{z}$ denotes the complex conjugate of $z$. 

We now define the unbounded operator $\left( \mathcal A, \mathcal D(\mathcal A ; \mathcal{Z})\right) $ in $\mathcal Z$ by$$\mathcal D(\mathcal A ; \mathcal{Z})=\left\lbrace \begin{pmatrix}
		\rho\\u\\S
	\end{pmatrix}\in \mathcal Z\: :\:u\in H_0^1(0,\pi), \left( -b\rho+\frac{1}{\rho_s}S \right)\in H^1(0,\pi) \right\rbrace  $$and
	\begin{equation}\label{def-A}
		\mathcal A=\begin{bmatrix}
			0\hspace{1mm}&-{\rho_s}\frac{d}{dx}\hspace{1mm}&0\vspace{2mm}\\-b\frac{d}{dx}&0&\frac{1}{\rho_s}\frac{d}{dx}\vspace{2mm}\\0&\frac{\mu}{\kappa}\frac{d}{dx}&-\frac{1}{\kappa}
		\end{bmatrix}.
	\end{equation}
	We introduce the input space $\mathcal{U} = \mathcal{Z}$ and the control operator $\mathcal{B} \in \mathcal{L}(\mathcal{U};\mathcal{Z})$ defined by 
	\begin{equation}\label{eqcontr}
		\mathcal{B} F  = \left( \mathbbm{1_{\mathcal{O}_{1}}} f_1, \mathbbm{1_{\mathcal{O}_{2}}} f_2,\mathbbm{1_{\mathcal{O}_{3}}} f_3 \right)^{\top}, \qquad F = (f_{1}, f_{2}, f_3)^{\top} \in \mathcal{U}.
	\end{equation}
	With the above introduced notations, the system  \eqref{eq2}-\eqref{bdd-ini} can be rewritten as 
	\begin{equation} \label{op-eqn}
		\dot{z}(t) = \mathcal{A} z(t) + \mathcal{B} F(t), \quad t\in (0,T), \qquad z(0) = z_{0},
	\end{equation}
	where we have set $z(t) = (\rho(t,\cdot), u(t, \cdot), S(t,\cdot))^{\top},$ $z_{0} =(\rho_{0}, u_{0}, S_0)^{\top}$ and \\
	$F(t) = (f_{1}(t, \cdot), f_{2}(t, \cdot), f_3 (t, \cdot))^{\top}.$ 
	
	The well-posedness of the system \eqref{eq2}-\eqref{bdd-ini} follows from the following proposition.
	\begin{prop} \label{pr:semigroup-z} 
		The operator $(\mathcal{A}, \mathcal{D}(\mathcal{A}; \mathcal{Z}))$ is the infinitesimal generator of a strongly continuous semigroup $\left\lbrace \mathbb{T}_{t} \right\rbrace_{t\geq 0} $ on $\mathcal{Z}.$
		Further, for any $F\in L^2(0,T; \mathcal{Z})$ and for any $z_0\in \mathcal{Z}$, \eqref{op-eqn} admits a unique solution $(\rho, u, S)\in C([0,T]; \mathcal{Z})$ with 
		\begin{equation}
			\|(\rho, u, S)\|_{C([0,T];\mathcal{Z})} \leqslant  C \Big(\|z_0\|_\mathcal{Z}+ \|F\|_{L^2(0,T; \mathcal{Z})}\Big).
		\end{equation}
	\end{prop}
	\begin{proof} The result follows from \cite[Theorem 3.8.4, Section 3.8, Chapter 3]{TW09} by showing that the operator $(\mathcal{A}, \mathcal{D}(\mathcal{A}; \mathcal{Z}))$ is maximal dissipative in $\mathcal{Z}$. \\
	
		\noindent {\textit{Step 1. Dissipativity:}} For all $\begin{pmatrix}
			\rho\\ u\\S
		\end{pmatrix}\in \mathcal{D}(\mathcal{A}; \mathcal{Z})$, we have
		\begin{align*}
			\left\langle \mathcal{A}\begin{pmatrix}
				\rho\\ u\\S
			\end{pmatrix},\begin{pmatrix}
				\rho\\ u\\S
			\end{pmatrix} \right\rangle_{\mathcal Z}=&\left\langle \begin{pmatrix}
				-{\rho_s}u'\\ (-b\rho+\frac{1}{\rho_s} S)'\\\frac{\mu}{\kappa}u'-\frac{1}{\kappa}S
			\end{pmatrix},\begin{pmatrix}
				\rho\\ u\\S
			\end{pmatrix} \right\rangle_{\mathcal Z}\\
			=&-b\int_{0}^{\pi}{\rho_s}u'\bar{\rho} \;\rd x + \rho_s\int_{0}^{\pi}\left( -b\rho+\frac{1}{\rho_s} S\right) '\bar{u} \;\rd x + \frac{\kappa}{\mu}\int_{0}^{\pi} \left( \frac{\mu}{\kappa}u'-\frac{1}{\kappa}S\right) \bar{S} \;\rd x\\
			=&-b\rho_s\int_0^\pi\Big(u'\bar{\rho}-\bar{u}'\rho\Big)\; \rd x- \int_0^\pi \Big(S\bar{u}'-\bar{S}u'\Big)\; \rd x -\frac{1}{\mu}\int_{0}^{\pi}|S|^2 \; \rd x,
		\end{align*}
	and thus
		\begin{align*}
			\text{Re } \left\langle \mathcal{A}\begin{pmatrix}
				\rho\\ u\\S
			\end{pmatrix},\begin{pmatrix}
				\rho\\ u\\S
			\end{pmatrix} \right\rangle_{\mathcal Z}&= -\frac{1}{\mu}\int_{0}^{\pi}|S|^2 \; \rd x \leq 0 .
		\end{align*}
		Therefore, $(\mathcal{A}, \mathcal{D}(\mathcal{A}; \mathcal{Z}))$ is dissipative in $\mathcal{Z}$.
		
		\noindent {\textit{Step 2. Maximality :}} We will show that $\mathrm{Range}\,(\mathcal{I}-\mathcal{A})=\mathcal{Z}$, i.e., for any $(f,g,h)\in\mathcal{Z}$, there exists a unique $(\rho,u,S)\in \mathcal{D}(\mathcal{A}; \mathcal{Z})$ such that
		\begin{equation}\label{max1}
			\begin{dcases}
				\rho + {\rho_s}u'= f & \mbox{ in } (0,\pi), \\
				u + b\rho' - \frac{1}{\rho_s} S'= g & \mbox{ in } (0,\pi), \\
				S - \frac{\mu}{\kappa}u' + \frac{1}{\kappa} S= h & \mbox{ in } (0,\pi), \\
				u(0)=u(\pi)=0.
			\end{dcases}
		\end{equation}
		After eliminating $\rho', S'$ from $\eqref{max1}_2$, the equation along with $\eqref{max1}_4$ is equivalent to : find $u\in H^1_0(0,\pi)$ such that
		\begin{equation*}
			u- \left( b{\rho_s}+\frac{\mu}{\rho_s(1+\kappa)}\right) u'' = g - b f' + \frac{\kappa}{\rho_s(1+\kappa)} h',
		\end{equation*}
		where the right hand side is in $H^{-1}(0,\pi)$. Since $\left( b{\rho_s}+\frac{\mu}{\rho_s(1+\kappa)}\right)>0$, the above elliptic equation admits a unique solution $u$ in $ H^1_0(0,\pi)$. Now \eqref{max1}$_1$ and \eqref{max1}$_3$ yield that there exit unique $\rho, S \in L^2 (0,\pi)$ such that $(\rho, u, S)$ satisfies \eqref{max1}. From \eqref{max1}$_2$, it follows that $\left( -b\rho+\frac{1}{\rho_s}S \right)\in H^1(0,\pi)$. Thus we get $(\rho,u,S)\in \mathcal{D}(\mathcal{A}; \mathcal{Z})$.
	\end{proof}

By duality argument, it is well-known that the controllability of a pair $(\mathcal{A}, \mathcal{B})$ can be characterized by the observability of the pair 
$(\mathcal{A}^*, \mathcal{B}^*)$, where $\mathcal{A}^*$ and $\mathcal{B}^*$ are the adjoint operators of $\mathcal{A}$ and $\mathcal{B}$, respectively. Hence, it is important to determine the adjoint operator $\mathcal{A}^*$:
	
	\begin{prop} \label{adj-op}
		The adjoint of  $(\mathcal{A}, \mathcal{D}(\mathcal{A} ; \mathcal{Z}))$ in $\mathcal{Z}$ is defined by 
		\begin{equation} \label{dom-A*}
			\mathcal{D}(\mathcal{A}^*;\mathcal{Z}) = \left\lbrace \begin{pmatrix}
				\sigma\\v\\\tilde{S}
			\end{pmatrix}\in \mathcal Z\: :\:v\in H_0^1(0,\pi), \left( b\sigma - \frac{1}{\rho_s}  \tilde S \right)\in H^1(0,\pi) \right\rbrace,
		\end{equation}
		and 
		\begin{equation} \label{op-A*}
			\mathcal{A}^* =  \begin{bmatrix}
				0\hspace{1mm}&{\rho_s}\frac{d}{dx}\hspace{1mm}&0\vspace{2mm}\\b\frac{d}{dx}&0&-\frac{1}{\rho_s}\frac{d}{dx}\vspace{2mm}\\0&-\frac{\mu}{\kappa}\frac{d}{dx}&-\frac{1}{\kappa}
			\end{bmatrix}.
		\end{equation}
		Moreover $(\mathcal{A}^*, \mathcal{D}(\mathcal{A}^*; \mathcal{Z}))$ is the infinitesimal generator of a strongly continuous semigroup $\left\lbrace \mathbb{T}^*_{t} \right\rbrace_{t\geq 0} $ on $\mathcal{Z}.$
	\end{prop}

	We study the well-posedness of the adjoint system with non-homogenous source terms and boundary data for future purpose. More precisely, we consider the following non-homogeneous system 
	\begin{equation} \label{adj-1}
		\begin{dcases}
			\partial_t\sigma-{\rho_s}\partial_xv= \zeta_{1} & \mbox{ in } (0,T) \times (0, \pi), \\
			\partial_tv-b\partial_x\sigma+\frac{1}{\rho_s} \partial_x\tilde{S}= \zeta_{2} &  \mbox{ in } (0,T) \times (0, \pi), \\
			\partial_t\tilde S+\frac{1}{\kappa}\tilde S+\frac{\mu}{\kappa} \partial_xv= \zeta_{3} &  \mbox{ in } (0,T) \times (0, \pi), \\
			v(t,0) = h_{0}(t), \quad   v(t, \pi) = h_{\pi}(t), &   \mbox{ in } (0,T), \\
			\sigma(0, x)  = \sigma^{0}(x), \qquad v(0, x) = v^{0}(x), \qquad \tilde{S}(0, x) = \tilde{S}^{0}(x), &  \mbox{ in } (0, \pi).
		\end{dcases}
	\end{equation}
	From the well-posedness of the adjoint operator $\mathcal{A}^{*},$ we get the following result. 
	\begin{prop} \label{adj-source}
		Let $T > 0.$ Then for any $(\sigma^{0}, v^{0}, \tilde{S}^{0} ) \in \mathcal{Z},$ $(\zeta_{1}, \zeta_{2}, \zeta_{3}) \in \Big(L^{2}(0,T;L^{2}(0,\pi))\Big)^{3}$ and $(h_{0}, h_{\pi}) \in \Big(H^{1}(0,T)\Big)^{2},$ the system \eqref{adj-1}
		admits a unique solution $(\sigma, v, \tilde{S}) \in C([0,T];\mathcal{Z})$  satisfying 
		\begin{multline*}
			\|(\sigma, v, \tilde{S})\|_{C([0,T];\mathcal{Z})}   \leqslant C \Big( \norm{(\sigma^{0}, v^{0}, \tilde{S}^{0})}_{\mathcal{Z}} + \norm{(\zeta_{1}, \zeta_{2}, \zeta_{3})}_{(L^{2}(0,T;L^{2}(0,\pi)))^{3}} + \norm{(h_{0}, h_{\pi})}_{(H^{1}(0,T))^2} \Big),
		\end{multline*}
		where the positive constant $C$ depends only on $T,\pi$ and the coefficients of the system.
	\end{prop}

As mentioned in the introduction, the controllability results for \eqref{eq2}-\eqref{bdd-ini} is expected in subspaces of $\mathcal{Z}$ satisfying \eqref{stressavg} and \eqref{denseavg}. So, we need to have that the restriction of the operator $\mathcal{A}$ on those subspaces of  $\mathcal{Z}$ also generates a semigroup. 
From \cite[Proposition 2.4.4, Section 2.4, Chapter 2]{TW09}, the following result holds. 
	\begin{prop}
		Let us define the spaces
		\begin{align}
			{\mathcal Z_m}=& L^2(0, \pi)\times L^2(0, \pi)\times L_m^2(0, \pi),\\
			{\mathcal Z_{m,m}}=& L_m^2(0, \pi)\times L^2(0, \pi)\times L_m^2(0, \pi)\label{defn_Z_mm}.
		\end{align}
		where the spaces ${\mathcal Z_m}, {\mathcal Z_{m,m}}$ are equipped with the inner product \eqref{innerproduct}.

Then $\mathcal Z_m$ is invariant under the semigroups $\mathbb{T}$ and $\mathbb{T}^*$. The restriction of $\mathbb{T}$ to $\mathcal Z_m$ is a strongly continuous
semigroup in $\mathcal Z_m$ generated by $(\mathcal{A}, \mathcal{D}(\mathcal{A};\mathcal{Z}_m))$, where $\mathcal{D}(\mathcal{A};\mathcal{Z}_m)=\mathcal{D}(\mathcal{A}, \mathcal{Z})\cap \mathcal{Z}_m$. Also, the restriction of $\mathbb{T}^*$ to $\mathcal Z_m$ is a strongly continuous semigroup in $\mathcal Z_m$  generated by $(\mathcal{A}^*, \mathcal{D}(\mathcal{A}^*;\mathcal{Z}_{m}))$, where $\mathcal{D}(\mathcal{A}^*;\mathcal{Z}_{m})=\mathcal{D}(\mathcal{A}^*, \mathcal{Z})\cap \mathcal{Z}_m$. 

Similarly, $\mathcal{Z}_{m,m}$ is invariant under the semigroups $\mathbb{T}$ and $\mathbb{T}^*$. 
The restriction of $\mathbb{T}$ to $\mathcal Z_{m,m}$ is a strongly continuous
semigroup in $\mathcal Z_{m,m}$ generated by $(\mathcal{A}, \mathcal{D}(\mathcal{A};\mathcal{Z}_{m,m}))$, where $\mathcal{D}(\mathcal{A};\mathcal{Z}_{m,m})=\mathcal{D}(\mathcal{A}, \mathcal{Z})\cap \mathcal{Z}_{m,m}$. Also, the restriction of $\mathbb{T}^*$ to $\mathcal Z_{m,m}$ is a strongly continuous semigroup in $\mathcal Z_{m,m}$  generated by $(\mathcal{A}^*, \mathcal{D}(\mathcal{A}^*;\mathcal{Z}_{m,m}))$, where $\mathcal{D}(\mathcal{A}^*;\mathcal{Z}_{m,m})=\mathcal{D}(\mathcal{A}^*, \mathcal{Z})\cap \mathcal{Z}_{m,m}$. 
	\end{prop}
	%%%%%%%%%%%%%%%%%%%%%%%%%%%%%%%%%%%%%%%%%%%%%%%%%%%%%%%%%%%%%%%%%%%%%%%%%%%%%%%%%%%%%%%%

	%%%%%%%%%%%%%%%%%%%%%%%%%%%%%%%%%%%%%%%%%%%%%%%%%%%%%%%%%%%%%%%%%%%%%%%%%%%%%%%%%%%%%%%%%%%%%%%

	\section{ Lack of null controllability}\label{Lack of null controllability}
	In this section we prove \cref{thm:lack-1-bis} and \cref{thm:lack-1-bis-velo}. We know that the null controllability of a pair $(\mathcal{A}, \mathcal{B})$ is equivalent to the final-state observability of the pair $(\mathcal{A}^{*}, \mathcal{B}^{*}).$  We recall the final state observability of  $(\mathcal{A}^{*}, \mathcal{B}^{*}):$
	\begin{defn}
		The pair $(\mathcal{A}^{*}, \mathcal{B}^{*})$ is final-state observable in the Hilbert space $\mathcal{Z}=(L^2(0,\pi))^3$ at time $T$ if there exists a positive constant $C_{T} > 0$ such that 
		\begin{equation*}
			\int_{0}^{T} \norm{\mathcal{B}^{*} \mathbb{T}_{t}^{*} z}_{\mathcal{U}}^{2} \; \rd t \geqslant C_{T} \norm{\mathbb{T}_{T}^{*}z}^{2}_{\mathcal{Z}},\quad \forall\,  z \in \mathcal{D}(\mathcal{A}^{*}),
		\end{equation*}
		where $\mathbb{T}^*$ is the $C^0$-semigroup generated by $(\mathcal{A}^{*}, \mathcal{D}(\mathcal{A}^{*}; \mathcal{Z}))$ in $\mathcal{Z}$. 
	\end{defn}
	For $(\mathcal{A}^*, \mathcal{D}(\mathcal{A}^*; \mathcal{Z}))$ defined in \eqref{dom-A*}-\eqref{op-A*} and $(\sigma^{0}, v^{0}, \tilde{S}^0 ) \in \mathcal{Z},$ we set 
	\begin{equation*}
		(\sigma(t), v(t), \tilde{S}(t)) = \mathbb{T}_{t}^{*} (\sigma^{0}, v^{0}, \tilde{S}^0) \qquad (t \geqslant 0),
	\end{equation*}
	where $\mathbb{T}^*$ is the $C^{0}$-semigroup generated by $(\mathcal{A}^*, \mathcal{D}(\mathcal{A}^*; \mathcal{Z}))$ on $\mathcal{Z}$. 
	In view of \cref{adj-op}, $(\sigma, v, \tilde{S})$ belongs to $C([0,T];\mathcal{Z})$ and satisfies:
	\begin{equation}\label{eqadj}
		\begin{dcases}
			\partial_t\sigma -{\rho_s}\pa_x v =0 & \mbox{ in } (0,T) \times (0, \pi), \\
			\partial_t v-ab \partial_{x} \sigma + \frac{1}{\rho_s} \partial_x\tilde{S}= 0 &  \mbox{ in } (0,T) \times (0, \pi),\\
			\partial_t\tilde{S}+\frac{1}{\kappa}\tilde{S} + \frac{\mu}{\kappa} \partial_x v= 0 &  \mbox{ in } (0,T) \times (0, \pi),\\
			v(t,0) = v(t, \pi) = 0 &  \mbox{ in } (0,T), \\
			\sigma(0,x)=\sigma^{0}(x), \quad v(0,x)=v^0(x), \quad \tilde{S}(0,x)=\tilde{S}^0(x) & \mbox{ in } (0,\pi).
		\end{dcases}
	\end{equation}

	Then the null controllability of the system \eqref{eq2}-\eqref{bdd-ini} is equivalent to the following observability inequality: 

\begin{prop}\label{thobs1}
Let $T>0$. 
\begin{enumerate}
\item The system \eqref{eq2}-\eqref{bdd-ini} is null controllable in $\mathcal{Z}$ at time $T>0$ using three controls $f_1$, $f_2$ and $f_3$ in $L^2(0,T;L^2(0,\pi))$ with supports in $\mathcal{O}_1$, $\mathcal{O}_2$ and $\mathcal{O}_3$ respectively, if and only if, for $T>0$, there exists a positive constant $C_T>0$  such that for any $(\sigma^{0}, v^{0}, \tilde{S}^0)\in \mathcal{Z}$, 
		$(\sigma, v, \tilde{S})$, the solution of \eqref{eqadj}, satisfies the following observability inequality:
		\begin{multline*}
			\int_{0}^{\pi} |\sigma(T,x)|^{2} \ \rd x \; + \; \int_{0}^{\pi} |v(T,x)|^{2} \ \rd x + \; \int_{0}^{\pi} |\tilde{S}(T,x)|^{2} \ \rd x \\
			\leqslant C_T \Big( \int_0^T \int_{\mathcal{O}_1} |\sigma(t,x)|^2\, \rd x\,\rd t + \int_0^T\int_{\mathcal{O}_2}|v(t,x)|^2\, \rd x\,\rd t + \int_0^T\int_{\mathcal{O}_3}|\tilde{S}(t,x)|^2\,  \rd x\,\rd t\Big). 
		\end{multline*}

\item Assume further $f_1=0=f_3$. 	
The system \eqref{eq2}-\eqref{bdd-ini} is null controllable in $\mathcal{Z}_{m,m}$ at time $T>0$ using a control $f_2$ in $L^2(0,T;L^2(0,\pi))$ with support in $\mathcal{O}_2$ acting only in the velocity equation, if and only if, for $T>0$, there exists a positive constant $C_T>0$  such that for any $(\sigma^{0}, v^{0}, \tilde{S}^0)\in \mathcal{Z}_{m,m}$, 
		$(\sigma, v, \tilde{S})$, the solution of \eqref{eqadj}, satisfies the following observability inequality:
		\begin{multline*}
			\int_{0}^{\pi} |\sigma(T,x)|^{2} \ \rd x \; + \; \int_{0}^{\pi} |v(T,x)|^{2} \ \rd x + \; \int_{0}^{\pi} |\tilde{S}(T,x)|^{2} \ \rd x 
			\leqslant C_T  \int_0^T\int_{\mathcal{O}_2}|v(t,x)|^2\, \rd x\,\rd t. 
		\end{multline*}
	\end{enumerate}
		
	\end{prop}
	
	Thus to prove \cref{thm:lack-1-bis} and \cref{thm:lack-1-bis-velo}, we will construct special solutions to the adjoint problem \eqref{eqadj} such that the observability inequality does not hold. For that in the next subsection we will construct Gaussian beam solutions for the adjoint operator. The construction of the Gaussian beam solutions follows the same technique given in \cite[Theorem 3.1]{ahamed2022lack}. Here we adapt the method of the construction to the case of the three coupled equations for this particular example. 
	
	\subsection{ Gaussian Beam construction}\label{secGB}
	Here we construct Gaussian beam for the following operator :
	\begin{equation}\label{eqn21}
		\mathcal{L}_1 \begin{pmatrix}
			\sigma \\ v\\\tilde{S}
		\end{pmatrix} = 
		\begin{pmatrix}
			\partial_t\sigma-{\rho_s}\partial_xv \\ 
			\partial_tv-b\partial_x\sigma+\frac{1}{\rho_s} \partial_x\tilde{S}\\
			\partial_t\tilde S+\frac{1}{\kappa}\tilde S+\frac{\mu}{\kappa} \partial_xv
		\end{pmatrix}, \quad\text{in }[0,T]\times \mathbb{R}.
	\end{equation}

	We have the following result. 
	
	\begin{thm} \label{thm:GB-L1}
		Let $T>0, x_{0} \in \mathbb{R}$  and $k \in \mathbb{N}.$ Let us set 
		\begin{equation}\label{eqGB01}
			\varphi(x)= \frac{i}{2} (x-x_0)^2+(x-x_0) \quad  ( x\in \R).
		\end{equation}
		Then there exist  $\eta, \vartheta, \Upsilon \in C^{1}([0,T];C^{2}_{b}(\mathbb{R})),$  a positive constant $C$, which may depend on $T$ but independent of $k,$ and a sequence of functions $(\sigma_{k}, v_{k}, \tilde{S}_k)_{k \in \mathbb{N}}$ 
		satisfying 
		\begin{equation} \label{reg-app-sol}
			\sigma_{k} \in C^{1}([0,T];C^{1}_{b}(\mathbb{R})), \quad v_{k} \in C^{1}([0,T];C^{1}_{b}(\mathbb{R})),\quad \tilde S_{k} \in C^{1}([0,T];C^{1}_{b}(\mathbb{R})),
		\end{equation}
		with $\sigma_{k}$, $v_{k}$ and $\tilde{S}_k$ in the form 
		\begin{align}
			& \sigma_k(t,x)= k^{-3/4}\partial_x\Big(e^{ik\varphi(x)}\eta(t,x)\Big),    \qquad \qquad \quad \qquad ( t\in [0,T] , x\in\R ), \label{ansatz-sigma} \\
			& v_k(t,x)= k^{-3/4} e^{i k \varphi(x)}\vartheta(t,x),  \qquad \quad\qquad \qquad ( t\in [0,T], \;  x\in \R)   \label{ansatz-v},\\
			& \tilde{S}_k(t,x)= k^{-3/4}\partial_x\Big(e^{ik\varphi(x)}\Upsilon(t,x)\Big),    \qquad \qquad \quad \qquad ( t\in [0,T], x\in\R) , \label{ansatz-tildeS} 
		\end{align}	
		such that  $\eta(t,x_0), \vartheta(t,x_0), \Upsilon(t,x_0)$ are non-zero for all $t\in [0,T]$ and the following holds:
		\begin{align}
			& \sup_{t\in [0,T]}   \norm{\mathcal{L}_1 \begin{pmatrix}
					\sigma_{k} \\ v_{k}\\\tilde{S}_k
				\end{pmatrix} (t,\cdot)}_{L^2(\mathbb{R})\times L^2(\mathbb{R})\times L^2(\mathbb{R})} \leqslant  Ck^{-1}, \label{sol-app} \\
			& \lim_{k\rightarrow \infty} \int_\R |\sigma_k(t,x)|^2\, dx = \sqrt{\pi}|\eta(t,x_0)|^2> 0,  \qquad \qquad (t \in [0,T]),  \label{est-sigma-R}  \\
			&  \sup_{t\in [0,T]} \, \int_{|x-x_0|> k^{-1/4}} |\sigma_k(t,x)|^2\,  \rd  x \leqslant C e^{-\sqrt{k}/2}, \label{est-concentrate} \\
			&  \sup_{ t \in [0,T]} \int_\R |v_k(t,x)|^2\, \rd x \leqslant Ck^{-2},  \label{est-v-R} \\
			& \lim_{k\rightarrow \infty} \int_\R |\tilde{S}_k(t,x)|^2\, dx = \sqrt{\pi}|\Upsilon(t,x_0)|^2> 0,  \qquad \qquad (t \in [0,T]),  \label{est-tildeS-R}  \\
			&  \sup_{t\in [0,T]} \, \int_{|x-x_0|> k^{-1/4}} |\tilde{S}_k(t,x)|^2\,  \rd  x \leqslant C e^{-\sqrt{k}/2}.\label{est-concentrate-tildeS}
		\end{align}
	\end{thm}
	\begin{proof}  
		
		\textit{ Step 1. Construction of $(\sigma_{k}, v_{k}, \tilde{S}_k):$}  
		
		Let $\varphi$ be as defined in \eqref{eqGB01} and let 
		for each $k\in \N,$ $(\sigma_{k}, v_{k}, \tilde{S}_k)$ be in the form \eqref{ansatz-sigma}-\eqref{ansatz-tildeS}. Our aim is to choose $\eta, \vartheta$ and $\Upsilon$ suitably so that \eqref{sol-app} - \eqref{est-concentrate-tildeS} holds. Plugging the above expressions of $\sigma_{k}$, $v_{k}$ and $\tilde{S}_k$ in  \eqref{eqn21} and after some standard computations, we obtain 
		\begin{equation} \label{app-L1}
			\mathcal{L}_1 \begin{pmatrix}
				\sigma_{k} \\ v_{k} \\ \tilde{S}_k
			\end{pmatrix} = k^{-3/4}
			\begin{pmatrix}
				\partial_x\Big(e^{ik\phi(x)}g_1\Big)\\  \partial_{xx}\Big(e^{ik\phi(x)}h_{1}\Big)+e^{ik\phi(x)}h_{0} \\ \partial_x\Big(e^{ik \phi(x)} l_{1}\Big)
			\end{pmatrix},
		\end{equation}
		where 
		\begin{flalign*} 
			%\label{exp-ci}
			g_{1} = \partial_t \eta - {\rho_s} \vartheta, \quad l_1=\partial_t \Upsilon+\frac{1}{\kappa}\Upsilon + \frac{\mu}{\kappa}\vartheta,
		\end{flalign*}
		\begin{flalign*} 
			%\label{exp-d2}
			h_{1} = -b \eta + \frac{1}{\rho_s}  \Upsilon,\quad  h_0=\partial_t \vartheta.
		\end{flalign*}

		Since we want $(\sigma_{k}, v_{k}, \tilde{S}_k)$ such that \eqref{sol-app} holds, we choose $\eta, \vartheta$ and $\Upsilon$ such that 
		\begin{equation} \label{GB001}
			g_{1} (t,x) =  h_{1}(t,x) = l_{1}(t,x) = 0  \mbox{ for all } t \in [0,T], x \in \mathbb{R}.
		\end{equation}
		The condition $h_{1} = 0$ implies that $\ds \Upsilon = b\rho_s \eta,$ and $g_{1}=0$ implies that $\ds \vartheta=\frac{1}{\rho_s}\partial_t \eta$. Using these in the expression of $l_{1}$ above, we obtain the following ODE for $\eta:$
		\begin{equation} \label{ODe-a0}
			\left( b\rho_s+\frac{\mu}{\kappa{\rho_s}}\right) \pa_t \eta(t,x) +\frac{b\rho_s}{\kappa} \eta(t,x)=0 \qquad \left(t \in [0,T], \; x \in \mathbb{R}\right).
		\end{equation}
		Let $\zeta \in C^\infty_c(\R)$ with $\zeta(x_0)\neq 0$. We choose 
		\begin{equation} \label{eq:a0}
			\eta(t,x) = e^{\omega_0 t}\zeta(x), \text{ where } \omega_0=\frac{-b{\rho_s}^2}{(b\kappa\rho_s^2+ \mu)}<0, \quad \left(t\in [0,T], x \in \mathbb{R}\right).
		\end{equation}
		With the above choice of $\eta,$ we take 
		\begin{flalign} \label{eq:b0}
			\Upsilon(t,x) = b\rho_s \eta(t,x), \quad \vartheta(t,x) &=\frac{1}{\rho_s}\partial_t \eta(t,x)   \quad \left(t\in [0,T], x \in \mathbb{R}\right).
		\end{flalign}
		It is easy to verify that $\eta, \vartheta, \Upsilon \in C^{2}([0,T];C^{2}_{b}(\mathbb{R}))$,
		so that $(\sigma_{k}, v_{k}, \tilde{S}_k)$ satisfies  \eqref{reg-app-sol}, and $h_{0} \in C([0,T];C_{b}(\mathbb{R}))$.

In the following steps, we indicate the proof of the estimates \eqref{sol-app}-\eqref{est-concentrate-tildeS}.

\textit{Step 2.}	
With the above choices of $\eta, \vartheta, \Upsilon$, \eqref{app-L1} is reduced to 
\begin{equation*}
			\mathcal{L}_1 \begin{pmatrix}
				\sigma_{k} \\ v_{k} \\ \tilde{S}_k
			\end{pmatrix} = k^{-3/4}
			\begin{pmatrix}
				0\\  e^{ik\phi(x)}h_{0} \\ 0
			\end{pmatrix},
		\end{equation*}	
and we obtain \eqref{sol-app} noting that 
$$ \int_\R \Big|k^{-3/4}e^{ik \phi(x)}h_{0}(t,x)\Big|^2 \ \rd x\le k^{-2}\sqrt{\pi}\|h_0\|^2_{L^\infty((0,T)\times \R)}, \quad \forall\, t\in[0,T].$$

The estimate \eqref{est-v-R} can be obtained similarly as the above estimate.

\textit{Step 3.}
Now to prove \eqref{est-concentrate}, using the expression of $\sigma_k, \, \phi, \, \phi'$, we note that for a generic positive constant $C$, independent of $k$ and $t$, 
$$ 
\begin{array}{l}
\ds
\int_{|x-x_0|> k^{-1/4}} |\sigma_k(t,x)|^2\,  \rd  x \leqslant C \Big(\int_{|x-x_0|>k^{-1/4}}k^{1/2}e^{-k(x-x_0)^2} [(x-x_0)^2+1]|\eta(t,x)|^2\, \rd x  \\
\ds \hspace{3cm} + \int_{|x-x_0|>k^{-1/4}}k^{-3/2}e^{-k(x-x_0)^2}|\partial_x \eta(t,x)|^2\, \rd x\Big), \quad \forall\, t\in[0,T],
\end{array}
$$
and then using the change of variable $z=\sqrt{\frac{k}{2}}(x-x_0)$ and $\ds \int_\R e^{-z^2}\, \rd z=\sqrt{\pi}$, $\ds \int_\R z^2e^{-z^2}\, \rd z=\frac{\sqrt{\pi}}{2}$, 
for all $k\in \mathbb{N}$ and $t\in [0,T]$, we have 
$$ 
\begin{array}{l}
\ds
\int_{|x-x_0|>k^{-1/4}}k^{1/2}e^{-k(x-x_0)^2} [(x-x_0)^2+1]|\eta(t,x)|^2\, \rd x \\[2.mm]
\ds 
\leqslant \sqrt{2}\|\eta\|_{L^\infty((0,T)\times \R)}\int_{|z|>\frac{k^{1/4}}{\sqrt{2}}} \Big(\frac{2}{k}z^2+1\Big)e^{-2z^2}\, \rd z\\[2.mm]
\ds
\leqslant C e^{-\frac{\sqrt{k}}{2}}\Big(\int_\R e^{-z^2}\, \rd z + 2\int_\R z^{2}e^{-z^2}\, \rd z\Big) \leq Ce^{-\frac{\sqrt{k}}{2}},
\end{array}
$$
and  similarly, for all $k\in \mathbb{N}$ and $t\in [0,T]$, we have 
$$ 
\begin{array}{l}
\ds
\int_{|x-x_0|>k^{-1/4}}k^{-3/2}e^{-k(x-x_0)^2}|\partial_x \eta(t,x)|^2\, \rd x \\
\ds
\leq \sqrt{2}k^{-2}\|\partial_x \eta\|_{L^\infty((0,T)\times \R)}\int_{|z|>\frac{k^{1/4}}{\sqrt{2}}}e^{-2z^2}\, \rd z\leq Ce^{-\frac{\sqrt{k}}{2}}.
\end{array}
$$		
From above estimates, \eqref{est-concentrate} follows. Similarly \eqref{est-concentrate-tildeS} can be proved. 

\textit{Step 4.}
Now to prove \eqref{est-sigma-R}, noting that 	$\ds k^{\frac{1}{2}}\int_\R e^{-k(x-x_0)^2}\, \rd x=\sqrt{\pi}$, we get 
$$ 
\begin{array}{l}
\ds
\int_\R |\sigma_k(t,x)|^2\, \rd x= k^{\frac{1}{2}}\int_\R e^{-k(x-x_0)^2}\Big|i\phi'(x)\eta(t,x)+\frac{1}{k}\partial_x \eta(t,x)\Big|^2\, \rd x\\
\ds \hspace{2cm} =\sqrt{\pi}|\eta(t,x_0)|^2 + R_k(t), \quad \forall\, t\in [0,T],
\end{array}
$$
 where for all $t\in [0,T]$, $R_k(t)$ is given by 
 $$\ds R_k(t)=k^{\frac{1}{2}}\int_\R e^{-k(x-x_0)^2}\Big(\Big|i\phi'(x)\eta(t,x)+\frac{1}{k}\partial_x \eta(t,x)\Big|^2-|\eta(t,x_0)|^2\Big)\, \rd x.$$
 
 Following the similar way as in Step 3, it can be show that 
 $$ 
\begin{array}{l} 
\ds  |R_k(t)|\le k^{\frac{1}{2}}\int_{|x-x_0|>k^{-1/4}} e^{-k(x-x_0)^2}\Big(\Big|i\phi'(x)\eta(t,x)+\frac{1}{k}\partial_x \eta(t,x)\Big|^2-|\eta(t,x_0)|^2\Big)\, \rd x   \\
 \ds \hspace{3cm}
 + k^{\frac{1}{2}}\int_{|x-x_0|\le k^{-1/4}} e^{-k(x-x_0)^2}\Big(\Big|i\phi'(x)\eta(t,x)+\frac{1}{k}\partial_x \eta(t,x)\Big|^2-|\eta(t,x_0)|^2\Big)\, \rd x \\
\ds \le C e^{-\frac{\sqrt{k}}{{2}}}+k^{\frac{1}{2}}\int_{|x-x_0|\le k^{-1/4}} e^{-k(x-x_0)^2}\Big(\Big|i\phi'(x)\eta(t,x)+\frac{1}{k}\partial_x \eta(t,x)\Big|^2-|\eta(t,x_0)|^2\Big)\, \rd x, \quad \forall\, t\in [0,T].
 \end{array}
$$
Using expressions of $\phi'$ and $\eta$, we can estimate 
$$ 
\begin{array}{l}
\Big|i\phi'(x)\eta(t,x)+\frac{1}{k}\partial_x \eta(t,x)\Big|^2-|\eta(t,x_0)|^2= \Big( -(x-x_0)\eta(t,x)+\frac{1}{k}\partial_x \eta(t,x)\Big)^2+ (\eta^2(t,x)-\eta^2(t,x_0))\\
\le 2\Big((x-x_0)^2\|\eta\|^2_{L^\infty((0,T)\times \R)}+  \frac{1}{k^2}\|\partial_x \eta\|^2_{L^\infty((0,T)\times \R)}+ (x-x_0)\|\eta\|_{L^\infty((0,T)\times \R)}\|\partial_x \eta\|_{L^\infty((0,T)\times \R)},
\end{array}
$$
and thus, we get for all $t\in [0,T]$, and for some generic positive constant $C$ independent of $k$ and $t$, 
$$ |R_k(t)|\le C \Big(e^{-\frac{\sqrt{k}}{{2}}}+k^{-\frac{1}{4}} \int_\R k^{\frac{1}{2}}e^{-k(x-x_0)^2}\, \rd x\Big)\le C \Big(e^{-\frac{\sqrt{k}}{{2}}}+\sqrt{\pi}k^{-\frac{1}{4}}\Big).$$
Hence, \eqref{est-sigma-R} holds. Similarly, \eqref{est-tildeS-R} can be proved.
\end{proof}

As a consequence of the above theorem, we get the following results.
\begin{lem} \label{trace-GB}
Let $x_{0} \in (0,\pi),$ and let  $(\sigma_{k}, v_{k}, \tilde{S}_k)_{k \in \mathbb{N}}$ be constructed as in \cref{thm:GB-L1}. Then for all $k ,$
\begin{gather*}
\norm{v_{k}(\cdot,0)}_{H^{1}(0,T)} \leqslant   Ck^{-3/4}, \qquad \norm{v_{k}(\cdot,\pi)}_{H^{1}(0,T)} \leqslant   C k^{-3/4}, 
\end{gather*}
where $C$ is a positive constant, which may depend on $T,$ but independent of $k.$
\end{lem}

	%%%%%%%%%%%%%%%%%%%%%%%%%%%%%%%%%%%%%%%%%%%%%%%%%%%%%%%%%%%%%%%%%%%%%
	\subsection{ Proof of \cref{thm:lack-1-bis} and \cref{thm:lack-1-bis-velo}}
	Now we are going to prove \cref{thm:lack-1-bis}.
	
	\begin{proof}[Proof of \cref{thm:lack-1-bis}]
		In view of  \cref{thobs1}, it is enough to show that, there exists a sequence of initial conditions $\left(\sigma^{0}_{k}, v^{0}_{k}, \tilde{S}^0_k \right)_{k\in \N}$ in $\mathcal{Z},$ such that, the corresponding solution $(\sigma_k, v_k, \tilde{S}_k)$ to the system \eqref{eqadj} satisfy the following estimates 
		\begin{equation*}
			\lim_{k\rightarrow \infty} \left(\int_0^T \int_{\mathcal{O}_1}|\sigma_k(t,x)|^2 \, \rd x \rd t + \int_0^T \int_{\mathcal{O}_2}|v_k(t,x)|^2 \, \rd x\rd t + \int_0^T \int_{\mathcal{O}_3}|\tilde{S}_k(t,x)|^2 \, \rd x\rd t\right)=0,
		\end{equation*}
		\begin{equation*}
			\lim_{k\rightarrow\infty} \left(\int_{0}^{\pi} |\sigma_k(T,x)|^{2} \ \rd x \; + \; \int_{0}^{\pi} |v_k(T,x)|^{2} \ \rd x + \; \int_{0}^{\pi} |\tilde{S}_k(T,x)|^{2} \ \rd x\right) \geqslant {A},
		\end{equation*}
		for some ${A}>0,$ independent of $k.$
		
		Since $(0,\pi)\backslash \overline{\mathcal{O}_{1}\cup \mathcal{O}_{3}}$ is a nonempty open subset of $(0,\pi)$, choose $x_0 \in (0,\pi)\backslash \overline{\mathcal{O}_{1}\cup \mathcal{O}_{3}}$ and $k_0\in\mathbb{N}$ such that
		\begin{equation}\label{x_0}
			\left\lbrace x : |x-x_0|< k_0^{-1/4}\right\rbrace\subset (0,\pi)\backslash \overline{\mathcal{O}_{1}\cup \mathcal{O}_{3}}.
		\end{equation}
		Let  $\left(\sigma_{k}^{\sharp}, v_{k}^{\sharp}, \tilde{S}_k^{\sharp}\right)_{k \in \mathbb{N}}$ be sequence of functions constructed in \cref{thm:GB-L1}.
		
		For $k\geq k_0$, let us define 
		\begin{equation*}
			h_{0,k}(t) := v^{\sharp}_k(t,0), \quad h_{\pi,k}(t) :=v^{\sharp}_{k}(t,\pi)  \qquad  ( t\in [0,T]),
		\end{equation*}
		and 
		\begin{align*}
			& \begin{pmatrix}
				\zeta_{1,k} \\ \zeta_{2,k} \\ \zeta_{3,k} 
			\end{pmatrix}(t,x) := \mathbbm{1}_{(0,\pi)}\mathcal{L}_1\begin{pmatrix}
				\sigma^{\sharp}_{k} \\  v^{\sharp}_{k} \\ \tilde{S}_k^{\sharp}
			\end{pmatrix} (t,x)  \quad  \left(  t \in [0,T], \ x\in [0,\pi] \right), 
		\end{align*}
		where $\mathcal{L}_1$ is the operator  defind in \eqref{eqn21}.
		
		For $k\geq k_0$, we consider the following system.
		\begin{equation}\label{eqadj-sharp}
			\begin{dcases}
				\partial_t \sigma^{\dagger}_{k} - {\rho_s}\partial_x v^{\dagger}_{k} = \zeta_{1,k} & \mbox{ in } (0,T) \times (0, \pi), \\
				\partial_t v^{\dagger}_{k}-b\partial_{x} \sigma^{\dagger}_{k} + \frac{1}{\rho_s} \partial_{x} \tilde{S}^{\dagger}_{k}= \zeta_{2,k} &  \mbox{ in } (0,T) \times (0, \pi),\\
				\partial_t \tilde{S}^{\dagger}_{k}+ \frac{1}{\kappa}\tilde{S}^{\dagger}_{k}+\frac{\mu}{\kappa} \partial_x v^{\dagger}_{k} = \zeta_{3,k} &  \mbox{ in } (0,T) \times (0, \pi),\\
				v^{\dagger}_{k} (t,0) = h_{0,k}(t), \quad  v^{\dagger}_{k} (t, \pi) = h_{\pi,k}(t) &  \mbox{ in } (0,T), \\
				\sigma^{\dagger}_{k} (0,x)= 0 , \quad v^{\dagger}_{k} (0,x) = 0 \quad  \tilde{S}^{\dagger}_{k}(0,x)= 0 & \mbox{ in } (0,\pi).
			\end{dcases}
		\end{equation} 
		From \cref{adj-source}, \eqref{sol-app} and \cref{trace-GB},   for $k\geq k_0$, the  system \eqref{eqadj-sharp} admits a unique solution $(\sigma^{\dagger}_{k}, v^{\dagger}_{k}, \tilde{S}^{\dagger}_{k}) \in C([0,T];\mathcal{Z})$ together with the estimate 
		\begin{equation} \label{est-dagger}
			\norm{(\sigma^{\dagger}_{k}, v^{\dagger}_{k}, \tilde{S}^{\dagger}_{k})}_{C([0,T];\mathcal{Z})}  \leqslant C k^{-\frac{3}{4}},
		\end{equation}
		where the positive constant $C$ is independent of $k.$ Now, we set 
		\begin{equation}\label{eq-final-construction}
			\sigma_{k} = \sigma^{\sharp}_{k} - \sigma^{\dagger}_{k}, \qquad v_{k} = v^{\sharp}_{k} - v^{\dagger}_{k},\qquad \tilde{S}_{k} = \tilde{S}^{\sharp}_{k} - \tilde{S}^{\dagger}_{k}.
		\end{equation}
		Then $(\sigma_{k}, v_{k}, \tilde{S}_k)$ satisfies the system \eqref{eqadj} with the initial data $\left(\sigma^{0}_{k}, v^{0}_{k}, \tilde{S}_k^{0}\right) = \left(\sigma^{\sharp}_{k}(0), v^{\sharp}_{k}(0), \tilde{S}^{\sharp}_{k}(0) \right).$ 
		From \eqref{x_0}, we get
		\begin{equation*}
			S=\{x\in \R \mid |x-x_0|< k^{-1/4}\}\subset (0,\pi)\text{ for }k\geq k_0,
		\end{equation*}
	and thus
	\begin{align*}
		\int_S |{\sigma}^{\sharp}_k(T,x)|^2\, \rd x = \int_\R |{\sigma}^{\sharp}_k(T,x)|^2\, \rd x  -\int_{|x-x_0|\ge k^{-1/4}}|{\sigma}^{\sharp}_k(T,x)|^2\, \rd x .
	\end{align*}
Now from \eqref{est-sigma-R}, \eqref{est-concentrate} and the above inequality, it follows that 
\begin{equation}\label{est_sigmaksharp}
	 \lim_{k\rightarrow \infty} \int_0^{\pi} |{\sigma^{\sharp}}_k(T,x)|^2\, \rd x \geqslant \lim_{k\rightarrow \infty} \int_S |{\sigma^{\sharp}}_k(T,x)|^2\, \rd x \geqslant  \frac{\sqrt{\pi}|\eta(T,x_0)|^2}{2},
\end{equation}
for the positive constant $\sqrt{\pi}|\eta(T,x_0)|^2$, same as in \eqref{est-sigma-R}, independent of $k$. Now using $\sigma_{k} = \sigma^{\sharp}_{k} - \sigma^{\dagger}_{k},$ \eqref{est_sigmaksharp} and \eqref{est-dagger},  we deduce that 
		\begin{align*}
			\lim_{k\rightarrow\infty} &\left(\int_{0}^{\pi} |\sigma_k(T,x)|^{2} \ \rd x \; + \; \int_{0}^{\pi} |v_k(T,x)|^{2} \ \rd x + \; \int_{0}^{\pi} |\tilde{S}_k(T,x)|^{2} \ \rd x\right) \\
			&\geqslant  \lim_{k\rightarrow\infty} \int_{0}^{\pi} |\sigma_k(T,x)|^{2} \ \rd x \geqslant {A}(T),
		\end{align*}
		for some positive constant ${A}(T)$. 
		Similarly,  from \eqref{est-concentrate}, \eqref{est-v-R}, \eqref{est-concentrate-tildeS} and \eqref{est-dagger}, it follows that 
		$$
		\lim_{k\rightarrow \infty} \Big(\int_0^T \int_{\mathcal{O}_1}|\sigma_k(t,x)|^2 \, \rd x \rd t+ \int_0^T \int_{\mathcal{O}_2}|v_k(t,x)|^2 \, \rd x \rd t + \int_0^T \int_{\mathcal{O}_3}|\tilde{S}_k(t,x)|^2 \, \rd x \rd t\Big)=0.$$
		This completes the proof of the theorem. 
	\end{proof}

\begin{proof}[Proof of \cref{thm:lack-1-bis-velo}] The proof of \cref{thm:lack-1-bis-velo} is similar to that of \cref{thm:lack-1-bis} given above. 
The main difference is that the construction of the sequence of initial conditions $\left(\sigma^{0}_{k}, v^{0}_{k}, \tilde{S}^0_k \right)_{k\in \N}$ has to be from $\mathcal{Z}_{m,m}$. 
This is assured because of the construction \eqref{ansatz-sigma}-\eqref{ansatz-tildeS} and the choice of $\zeta$ in \eqref{eq:a0} with support inside $(0,\pi)$, i.e., 
$\zeta(0)=0=\zeta(\pi)$. Thus, noting that in the proof of  \cref{thm:lack-1-bis}, $\ds \int_0^\pi \sigma_{k}^{\sharp}(0, x)\, \rd x=0$ and $\int_0^\pi \tilde{S}^{\sharp}(0,x)\, \rd x=0$, from \eqref{eq-final-construction}, it follows that 
$$ \int_0^\pi \sigma^{0}_{k}(x)\, \rd x= 0 = \int_0^\pi \tilde{S}^{0}_{k}(x)\, \rd x, \quad \forall\, k\in \mathbb{N}.$$
Hence, the sequence $\left(\sigma^{0}_{k}, v^{0}_{k}, \tilde{S}^0_k \right)_{k\in \N}$ belongs to $\mathcal{Z}_{m,m}$. The rest of the proof follows similarly to that of \cref{thm:lack-1-bis}. 
\end{proof}

%%%%%%%%%%%%%%%%%%%%%%%%%%%%%%%%%%%%%%%%%%%%%%%%%%%%%%%%%%%%%%%%%%%%%%%%%%
%%%%%%%%%%%%%%%%%%%%%%%%%%%%%%%%%%%%%%%%%%%%%%%%%%%%%%%%%%%%%%%%%%%%%%%%%%%%%%%%%%	

	\section{ Spectral analysis of the linearized operator}\label{secspec}
	{The rest of the article depends on the spectral analysis of the linearized operator. In this section, we will discuss the behaviour of the spectrum of $\mathcal{A}$ on $\mathcal{Z}_m$. 	We define a Fourier basis $\left\lbrace \phi_{0,1}, \phi_{0,3},\phi_{n,l},\: l=1,2,3\right\rbrace_{n\geq 1} $ in $\mathcal{Z}$ as follows :
		\begin{align}\label{equ-Phi_n}
			&\phi_{0,1}(x)=\frac{1}{\sqrt{b\pi}}\begin{pmatrix}
				1\\0\\0
			\end{pmatrix},\quad\phi_{0,3}(x)=\sqrt{\frac{\mu}{\kappa\pi}}\begin{pmatrix}
				0\\0\\1
			\end{pmatrix},\quad\phi_{n,1}(x)={\sqrt{\frac{2}{b\pi}}}\begin{pmatrix}
				\cos(nx)\\0\\0
			\end{pmatrix},\notag\\
			&\phi_{n,2}(x)={\sqrt{\frac{2}{\rho_s\pi}}}\begin{pmatrix}
				0\\\sin(nx)\\0
			\end{pmatrix},\quad\phi_{n,3}(x)={\sqrt{\frac{2\mu}{\kappa\pi}}}\begin{pmatrix}
				0\\0\\\cos(nx)
			\end{pmatrix},\quad n\geq 1.
		\end{align}
		
We note that $\mathcal{Z}_m=  \mathrm{span}\, \{\phi_{0,1}\} \oplus_{n\in \mathbb{N}}\, \mathrm{span}\, \{\phi_{n,\ell}\mid \ell=1,2,3\}$, an orthogonal sum of the finite dimensional subspaces.
		
		\begin{prop}\label{spectrumofA}
			The spectrum of the operator $(\mathcal{A}, \mathcal{D}(\mathcal{A}; \mathcal{Z}_m))$ consists of $0$ and three sequences $\lambda_n^1$, $\lambda_n^2$ and $\lambda_n^3$ of eigenvalues. Furthermore we have the following properties:
			\begin{itemize}
				\item[$(a)$] For each $n$, there exists at least one eigenvalue, namely $\lambda_n^1$, is real and lies between $-\frac{1}{\kappa}$ and $0$.
				
				\item[$(b)$] All the eigenvalues of $\mathcal{A}$ have negative real part.
				
			 \item[$(c)$] As $n\rightarrow\infty$, $\lambda_n^1 \rightarrow w_0$, where 
				\begin{equation}\label{def-w_0}
					w_0=-\frac{b{\rho_s^2}}{ (\mu+ \kappa b{\rho_s^2})}.
				\end{equation}
				The other two roots $\lambda_n^2$ and $\lambda_n^3$ behave asymptotically as
				\begin{equation}\label{est-lambda_n2R}
					-\frac{1}{2}\left( w_0+\frac{1}{\kappa}\right)\pm in\sqrt{(b\rho_s+\frac{\mu}{\kappa\rho_s})}+\mathcal{O}\left(\frac{1}{\sqrt{2}n} \right).
				\end{equation}
				
				\item[$(d)$] If $n\neq m$, then $\lambda_n^\ell \neq \lambda_m^j$, for $\ell, j=1,2,3.$ 
				
				\item[$(e)$] The multiple roots $\lambda_n^\ell=\lambda_n^j$, for $\ell,j=1,2,3,$ can occur only for finitely many $n$.
			\end{itemize}
			
		\end{prop} 
		\begin{proof}
We see that $\mathcal{A}\phi_{0,1}=0$ and $ \phi_{0,1}\neq 0$. Thus $0$ is an eigenvalue of $\mathcal{A}$.

For each $n\in \mathbb{N}$, since the subspace given by $\mathrm{span}\, \{\phi_{n,\ell}\mid \ell=1,2,3\}$ is invariant under $\mathcal{A}$, to obtain the eigenvalues of $\mathcal{A}$, we compute the eigenvalues of $\mathcal{A}$ restricted to this subspace for each $n\in \mathbb{N}$. 
For each $n\in \mathbb{N}$, the eigenvalues of $\mathcal{A}$ restricted to  the $\mathrm{span}\, \{\phi_{n,\ell}\mid \ell=1,2,3\}$ are given by the roots of the characteristic equation
			\begin{equation}\label{equ-char}
				\mathcal{F}_n(\lambda) : =	\lambda^3+\left( \frac{1}{\kappa}\right)\lambda^2+\left(\frac{\mu}{\kappa\rho_s}n^2+b {\rho_s}n^2  \right)\lambda+\left( \frac{b{\rho_s}}{\kappa}n^2\right)=0.
			\end{equation}	
			From the relation between roots and the coefficients of $\mathcal{F}_n$, we get
			
			\begin{equation}\label{eq2.6}
				\left.
				\begin{aligned}
					&\lambda_n^1+\lambda_n^2+\lambda_n^3=-\frac{1}{\kappa},\\
					&\lambda_n^1\lambda_n^2+\lambda_n^2\lambda_n^3+\lambda_n^3\lambda_n^1=\left(\frac{\mu}{\kappa\rho_s}+b{\rho_s} \right)n^2,\\
					&\lambda_n^1\lambda_n^2\lambda_n^3=- \left( \frac{b{\rho_s}}{\kappa}\right)n^2.
				\end{aligned}
				\right\}
			\end{equation}
			\\
			$(a)$ Note that, for each $n\in\mathbb{N}$,       $\mathcal{F}_n(0)>0$ and $\mathcal{F}_n(-1/\kappa)<0$. Therefore, for each $n$, there exists at least one eigenvalue $\lambda_n^1$  between $-1/\kappa$ and $0$.\\
			$(b)$ From \eqref{eq2.6}$_1$, \eqref{eq2.6}$_3$ and using $(a)$, we get
			\begin{equation}\label{negative}
				\lambda_n^2+\lambda_n^3< 0,\quad \lambda_n^2\lambda_n^3 >0 .
			\end{equation}
		Here $\lambda_n^2$ and $\lambda_n^3$ are both real or complex conjugate. In both the cases using \eqref{negative}, we get  $\text{Re } \lambda_n^2<0$ and $\text{Re } \lambda_n^3<0$.\\
			
			$(c)$ Rewriting \eqref{equ-char} as $\lambda=-n^2\left[ \frac{b\rho_s}{\lambda}+\frac{\mu}{\kappa\rho_s\left(\lambda+\frac{1}{\kappa} \right) }\right]$,
			 and using from $(a)$ that $\left\lbrace \lambda_n^1\right\rbrace $ is bounded for all $n\in \mathbb{N}$, it can be derived that $\lambda_n^1$ is convergent to $\omega_0$ as $n\rightarrow \infty$, where $\omega_0$ is given by \eqref{def-w_0}.
		
	Now from $\eqref{eq2.6}_1$ and $\eqref{eq2.6}_2$, for large $n$, we get 
			\begin{align}
				\text{Re } \lambda_n^l=&-\frac{\lambda_n^1}{2}-\frac{1}{2\kappa},\label{equ-relambdan1}\\
				\text{Im } \lambda_n^l=&(-1)^l\sqrt{3/4\left(\lambda_n^1 \right)^2+1/\left(2 \kappa\right)\lambda_n^1-1/\left( 4\kappa^2\right)+ \left(\frac{\mu}{\kappa\rho_s}+b{\rho_s} \right)n^2    },\text{ for }l=2,3.\label{equ-imlambdan}
			\end{align}
Using \eqref{equ-relambdan1}, \eqref{equ-imlambdan} and the convergence of $\lambda_n^1\rightarrow \omega_0$ as $n\rightarrow \infty$, we get \eqref{est-lambda_n2R}.\\

			$(d)$ It can be proved using a contradiction argument. Suppose that $\lambda_n^i=\lambda_m^j=\alpha$ (say) for $n\neq m$. Then $\mathcal{F}_n(\alpha)=\mathcal{F}_m(\alpha)=0$, which implies $\alpha=-\frac{b{\rho_s^2}}{ (\mu+ \kappa b{\rho_s^2})}$. This is a contradiction because with this $\alpha$, 
			$\mathcal{F}_n(\alpha)=\frac{\mu b^2\rho_s^4}{\kappa\left(\mu+ \kappa b{\rho_s^2} \right)^3 }\neq 0$. 
			Hence if $n\neq m$, $\lambda_n^\ell \neq \lambda_m^j$, for all $\ell, j=1,2,3$.

			$(e)$ It follows from $(a)$ and $(c)$.
		\end{proof}
		\begin{figure}[!tbp]
			\centering
			\subfloat[Spectrum of the operator $\mathcal{A}$]{\includegraphics[width=.8\textwidth]{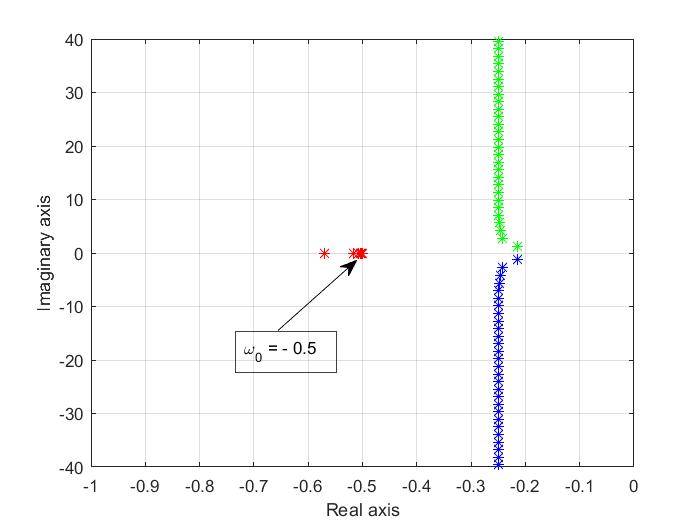}\label{fig:f1}}
			\hfill
			
			\caption{Eigenvalues of $\mathcal A$ in the complex plane for $n$ varies from $1$ to $50$ when $\mu=\rho_s=b=1$ and $\kappa=1$. Then $w_0=-0.5$. Red colour represents $\lambda_n^1$, green colour represents $\lambda_n^2$ and blue colour represents $\lambda_n^3$.}
		\end{figure}

		Now onwards, we assume the following :
		\begin{equation}\label{simpleeigen}
			\text{The spectrum of $\mathcal{A}$ has simple eigenvalues on $\mathcal{Z}_m$.}\tag{$\mathcal{H}$}
		\end{equation}
		
		Now we define the family $\left\lbrace \xi_0\right\rbrace \cup \left\lbrace \xi_{n,l}\: |\: 1\leq l\leq 3, n\in\mathbb{N}\right\rbrace$ as follows : 
		
		We choose a normalized eigenfunction of $\mathcal{A}$ for the eigenvalue $\lambda_0=0$ defined by
		\begin{equation}
			\xi_0=\frac{1}{\sqrt{b\pi}}\begin{pmatrix}
				1\\0\\0
			\end{pmatrix},
		\end{equation}
		and for the eigenvalue $\lambda_n^l$, the normalized eigenfunction defined by
		\begin{equation}\label{equ-xi_n}
			\xi_{n,l}=\frac{1}{\theta_{n,l}}\begin{pmatrix}
				-\cos(nx)\vspace{1mm}\\\frac{\lambda_n^l}{{\rho_s}n}\sin(nx)\vspace{1mm}\\\frac{\mu\lambda_n^l}{\rho_s\left( 1+\kappa\lambda_n^l\right) }\cos(nx)
			\end{pmatrix},\: l\in \left\lbrace 1,2,3\right\rbrace,\: n\in\mathbb{N},
		\end{equation}
		where 
		\begin{equation}\label{equ-theta_nl}
			\theta_{n,l}=\sqrt{\frac{\pi}{2}\left(b+\frac{\left| \lambda_n^l\right|^2 }{\rho_sn^2} + \frac{\kappa\mu\left| \lambda_n^l\right|^2 }{\rho_s^2\left| 1+\kappa\lambda_n^l\right|^2}\right) }.
		\end{equation} $\left| 1+\kappa\lambda_n^l\right|\neq 0  \left( l=1,2,3\right) $, as $-1/\kappa$ is not a root of the characteristic polynomial \eqref{equ-char}.\\
		Similarly, we choose $\left\lbrace \xi^*_0\right\rbrace \cup \left\lbrace \xi^*_{n,l}\: |\: 1\leq l\leq 3, n\in\mathbb{N}\right\rbrace$ as follows : 
		
		$\xi^*_0$ is an eigenfunction of $\mathcal{A}^*$ with the eigenvalue $\lambda_0=0$ defined by
		\begin{equation}\label{equ-xi_0^*}
			\xi^*_0=\frac{1}{\sqrt{b\pi}}\begin{pmatrix}
				1\\0\\0
			\end{pmatrix},
		\end{equation}
		and for the eigenvalue $\overline{\lambda_n^l}$, the eigenfunction defined by
		\begin{equation}\label{equ-xi_n^*}
			\xi^*_{n,l}=\frac{1}{\psi_{n,l}}\begin{pmatrix}
				\cos(nx)\vspace{1mm}\\\frac{\overline{\lambda_n^l}}{{\rho_s}n}\sin(nx)\vspace{1mm}\\-\frac{\mu\overline{\lambda_n^l}}{\rho_s\left( 1+\kappa\overline{\lambda_n^l}\right) }\cos(nx)
			\end{pmatrix},\: l\in \left\lbrace 1,2,3\right\rbrace,\: n\in\mathbb{N},
		\end{equation}
		where
		\begin{equation}\label{equ-psi_nl}
			\psi_{n,l}=\frac{\sqrt{\frac{\pi}{2}}\left(-b+\frac{\left(  \overline{\lambda_n^l}\right)^2 }{\rho_sn^2} - \frac{\mu\kappa\left(  \overline{\lambda_n^l}\right)^2 }{\rho_s^2\left(  1+\kappa\overline{\lambda_n^l}\right)^2}\right)  }{\sqrt{\left(b+\frac{\left| \lambda_n^l\right|^2 }{\rho_sn^2} + \frac{\mu\kappa\left| \lambda_n^l\right|^2 }{\rho_s^2\left| 1+\kappa\lambda_n^l\right|^2}\right) }}.
		\end{equation} 
		Here for all $n\in\mathbb{N}$, $\psi_{n,l}\neq 0$ for $l=1,2,3 $, because of the assumption \eqref{simpleeigen}.		

		The choice of this family of eigenfunctions of $\mathcal{A}$ and $\mathcal{A}^*$ ensures the following lemma.
		
		\begin{lem}\label{biorthonormality}
			Under the assumption \eqref{simpleeigen}, the families $\left\lbrace \xi_0\right\rbrace \cup \left\lbrace \xi_{n,l}\: |\: 1\leq l\leq 3, n\in\mathbb{N}\right\rbrace$ and $\left\lbrace \xi^*_0\right\rbrace \cup \left\lbrace \xi^*_{n,l}\: |\: 1\leq l\leq 3, n\in\mathbb{N}\right\rbrace$ satisfy the following bi-orthonormality relations :
			\begin{align}
				&\left\langle \xi_0, \xi^*_{k,p}\right\rangle_{\mathcal{Z}}=0,\quad \left\langle \xi_{n,l}, \xi^*_{k,p}\right\rangle_{\mathcal{Z}}=\delta_k^n\delta_p^l,\: l,p\in \left\lbrace 1,2,3\right\rbrace,\: k,n\in\mathbb{N},\label{biortho}\\
				& \left\langle \xi_0, \xi^*_{0}\right\rangle_{\mathcal{Z}}=1,\quad \left\langle \xi_{n,l}, \xi^*_{0}\right\rangle_{\mathcal{Z}}=0,\: l\in \left\lbrace 1,2,3\right\rbrace,\: n\in\mathbb{N},
			\end{align}
			where 
			$$
			\delta_n^l=\begin{cases}
				0, \: n\neq l\\
				1,\: n=l.
			\end{cases}
			$$
			
			Moreover, the asymptotic behaviors of $\theta_{n,l}$ and $\psi_{n,l}$ are as follows :
			\begin{equation}\label{cgs-theta_nl}
				\begin{aligned}
					&\left|\theta_{n,l} \right|, \left|\psi_{n,l} \right|\rightarrow  \sqrt{\frac{\pi}{2}\left( b+\frac{\kappa b^2\rho_s^2}{\mu}\right)   }, \text{ for }l=1,\text{ as }n\rightarrow \infty,\\
					& \left|\theta_{n,l} \right|,\left|\psi_{n,l} \right|\rightarrow  \sqrt{{\pi}\left(b+\frac{\mu}{\kappa\rho_s^2} \right)  }, \text{ for }l=2,3, \text{ as }n\rightarrow \infty.
				\end{aligned}
			\end{equation}
		\end{lem}	
		\subsection{ Riesz Basis.}\label{secRiesz}	
		
		The aim of this subsection is to show that the eigenfunctions of $\mathcal{A}$ form a Riesz basis in $\mathcal{Z}_m$.
		Recall from \eqref{equ-Phi_n} that $\left\lbrace \phi_{0,1}, \phi_{n,1},\phi_{n,2}, \phi_{n,3} | n\in\mathbb{N}\right\rbrace $  forms an orthonormal basis in $\mathcal{Z}_m$.

		For the definition of Riesz basis in $\mathcal{Z}_m$, we refer \cite[Definition 2.5.1, Chapter 2, Section
		2.5]{TW09}. In other words, the sequence $\left\lbrace \xi_0\right\rbrace \cup \left\lbrace \xi_{n,l}\: |\: 1\leq l\leq 3\right\rbrace_{ n\in\mathbb{N}}$ forms a Riesz basis in $\mathcal{Z}_m$ if there exists an
		invertible operator $Q\in\mathcal{L}\left( \mathcal{Z}_m\right) $ such that
		\begin{equation*}
			Q\phi_{0,1}=\xi_0,\: Q\phi_{n,1}=\xi_{n,1},\:Q\phi_{n,2}=\xi_{n,2},\:Q\phi_{n,3}=\xi_{n,3},\:\forall n\in\mathbb{N}.
		\end{equation*}

		\begin{lem}\label{lem-bais_representation}
			Assume \eqref{simpleeigen}. Let us set $\mathcal{B}_n^1=\left\lbrace \phi_{n,l}\: | l=1,2,3\right\rbrace $
			and $\mathcal{B}_n^2=\left\lbrace \xi_{n,l}\: | l=1,2,3\right\rbrace $ for all $n\in\mathbb{N}$, where $\left\lbrace \phi_{n,l}\right\rbrace_{1\leq l\leq 3} $ and $\left\lbrace \xi_{n,l}\right\rbrace_{1\leq l\leq 3} $ are defined as in \eqref{equ-Phi_n} and \eqref{equ-xi_n}. Let $z_n$ be expressed in the basis $\mathcal{B}_n^1$ and $\mathcal{B}_n^2$ as follows :
			\begin{equation}\label{equ-z_n}
				z_n=\sum\limits_{p=1}^{3}c_{n,p}\phi_{n,p}=\sum\limits_{l=1}^{3}d_{n,l}\xi_{n,l},
			\end{equation}
			and let the operator $\Gamma_n\in \mathbb{C}^{3\times 3}$ be defined by
			\begin{equation}\label{wedge}
				\Gamma_n \begin{pmatrix}
					c_{n,1}\\c_{n,2}\\c_{n,3}
				\end{pmatrix}= \begin{pmatrix}
					d_{n,1}\\d_{n,2}\\d_{n,3}
				\end{pmatrix}.
			\end{equation}
			Then there exist positive constants $\alpha$ and $\beta$ independent of $n$ such that
			\begin{equation}
				\left\| \Gamma_n\right\| < \alpha, \quad  \left\| \Gamma_n^{-1}\right\| < \beta,\text{ for all }n\in \mathbb{N}.
			\end{equation}
		\end{lem}

Using the \cref{lem-bais_representation} and from the second part of  \cite[Proposition 2.5.3, Chapter 2, Section 2.5]{TW09}, we can see that the eigenfunctions of $\mathcal{A}$, $\big\lbrace\xi_0 , \xi_{n,l}\: | 1\leq l\leq 3, \: n\in \mathbb{N}\big\rbrace $ forms a Riesz basis in $\mathcal{Z}_m$. Also using \cite[Proposition 2.8.6, Chapter 2, Section 2.8]{TW09}, \cite[Definition 2.6.1, Chapter 2, Section 2.6]{TW09} and  \cref{biorthonormality}, we obtain
		$\left\lbrace\xi_0^* , \xi^*_{n,l}\: | 1\leq l\leq 3, \: n\in \mathbb{N}\right\rbrace $ also forms a Riesz basis in $\mathcal{Z}_m$.
		\begin{prop}\label{basis_representation}
			Assume \eqref{simpleeigen} holds. Then any $z\in{\mathcal{Z}_m}$ can be uniquely represented as
			\begin{equation}\label{expansion_z}
				z=\sum\limits_{n=1}^{\infty}\sum\limits_{l=1}^{3}\left\langle z, \xi^*_{n,l} \right\rangle_{\mathcal{Z}} \xi_{n,l}+\left\langle z, \xi^*_{0} \right\rangle_{\mathcal{Z}} \xi_{0}. 
			\end{equation}
			There exist positive numbers $C_1$ and $C_2$ such that
			\begin{equation}\label{equ-est-z_0}
				C_1\left( \left| \left\langle z, \xi^*_{0} \right\rangle\right|^2+\sum_{n=1}^{\infty}\sum_{l=1}^{3}\left| \left\langle z, \xi^*_{n,l} \right\rangle\right|^2 \right) \leq \left\|z \right\|^2_{{\mathcal{Z}_m}}\leq C_2\left(\left| \left\langle z, \xi^*_{0} \right\rangle\right|^2+ \sum_{n=1}^{\infty}\sum_{l=1}^{3}\left| \left\langle z, \xi^*_{n,l} \right\rangle\right|^2 \right). 
			\end{equation}

			Similarly, any $z\in\mathcal{Z}_m$ can be uniquely represented in the basis $\left\lbrace\xi_0^* \right\rbrace \cup \left\lbrace \xi^*_{n,l}\: | 1\leq l\leq 3, \: n\in \mathbb{N}\right\rbrace $ by
			\begin{equation}\label{expansion_z-adj}
				z=\sum\limits_{n=1}^{\infty}\sum\limits_{l=1}^{3}\left\langle z, \xi_{n,l} \right\rangle_{\mathcal{Z}} \xi^*_{n,l}+\left\langle z, \xi_{0} \right\rangle_{\mathcal{Z}} \xi^*_{0}, 
			\end{equation}
			and
			\begin{equation}\label{equ-est-adj-z_0}
				\frac{1}{C_2}\left( \left| \left\langle z, \xi_{0} \right\rangle\right|^2+\sum_{n=1}^{\infty}\sum_{l=1}^{3}\left| \left\langle z, \xi_{n,l} \right\rangle\right|^2 \right) \leq \left\|z \right\|^2_{{\mathcal{Z}_m}}\leq \frac{1}{C_1}\left(\left| \left\langle z, \xi_{0} \right\rangle\right|^2+ \sum_{n=1}^{\infty}\sum_{l=1}^{3}\left| \left\langle z, \xi_{n,l} \right\rangle\right|^2 \right). 
			\end{equation}
		\end{prop}
		
		As a consequence of the above proposition, from \cite[Remark 2.6.4, Chapter 2, Section
		2.6]{TW09}, we get the following result of spectrum of $\mathcal{A}$ and $\mathcal{A}^*$.
		\begin{thm}
			The spectrum of the operator $(\mathcal{A}, \mathcal{D}(\mathcal{A}; \mathcal{Z}_m))$ is given by 
			$$\sigma(\mathcal{A})=\left\lbrace \lambda_0=0, \lambda_n^j,\: n\in\mathbb{N},\: j=1,2,3\right\rbrace \cup \left\lbrace w_0\right\rbrace \text{ in }\mathbb{C},$$
			i.e, the closure of $\{\lambda_0=0, \lambda_n^j,\: n\in\mathbb{N},\: j=1,2,3\}$, the set of eigenvalues of $\mathcal{A}$, in $\mathbb{C}$. 
			Here, $w_0$ is the same as in \eqref{def-w_0}.
			
			Similarly $\sigma(\mathcal{A}^*)=\left\lbrace \lambda_0, \lambda_n^j,\: n\in\mathbb{N},\: j=1,2,3\right\rbrace \cup \left\lbrace w_0\right\rbrace .$
		\end{thm}

	\subsection{ Projections and the projected system}\label{secproj}
	Let us define the following finite dimensional spaces
	\begin{equation}\label{def-Zn}
		\mathbf{Z}_0=\text{span} \left\lbrace \xi_{0} \right\rbrace,\:\mathbf{Z}_n=\text{span} \left\lbrace \xi_{n,l},\: l=1,2,3\right\rbrace , n\geq 1.
	\end{equation}
	
	\cref{basis_representation} gives that the space ${\mathcal Z_m}$ is  the orthogonal sum of the subspaces $\left\lbrace  \mathbf{Z}_n\right\rbrace_{n\geq 0} $, i.e., 
	\begin{equation*}
		\mathcal{Z}_m= \oplus_{n=0}^{\infty}\mathbf{Z}_n.
	\end{equation*}
	In view of \eqref{expansion_z}, let us define the projection $\pi_n\in \mathcal{L}\left(\mathcal{Z}_m \right) $ with Range $\pi_n\subset \mathbf{Z}_n$ by
	\begin{equation}\label{def-pi}
		\pi_n z=  \sum\limits_{l=1}^{3}\left\langle z, \xi^*_{n,l} \right\rangle \xi_{n,l},\quad \forall\, n\geq 1,\quad 
		\pi_0 z = \left\langle z, \xi^*_{0} \right\rangle \xi_{0}, \quad \forall\, z\in \mathcal{Z}_m.
	\end{equation}
	
Similarly, defining 
\begin{equation}\label{def-adjZn}
		\mathbf{Z}^*_0=\text{span} \left\lbrace \xi^*_{0} \right\rbrace,\:\mathbf{Z}^*_n=\text{span} \left\lbrace \xi^*_{n,l},\: l=1,2,3\right\rbrace , n\geq 1,
	\end{equation}
we get $\mathcal{Z}_m= \oplus_{n=0}^{\infty}\mathbf{Z}^*_n$. 
In view of \eqref{expansion_z-adj}, the projection $\pi^*_n\in \mathcal{L}\left(\mathcal{Z}_m \right) $ with Range $\pi^*_n\subset \mathbf{Z}^*_n$ by
	\begin{equation}\label{def-adj-pi}
		\pi^*_n z=  \sum\limits_{l=1}^{3}\left\langle z, \xi_{n,l} \right\rangle \xi^*_{n,l},\quad \forall\, n\geq 1,\quad 
		\pi^*_0 z = \left\langle z, \xi_{0} \right\rangle \xi^*_{0}, \quad \forall\, z\in \mathcal{Z}_m.
	\end{equation}

	Now, let $f_2=0, f_3=0, f_1=f$ and $\mathcal{O}_1= (0,\pi)$. Then the system  \eqref{eq2}-\eqref{bdd-ini} can be rewritten as 
	\begin{equation} \label{opp-eqn}
		\dot{z}(t) = \mathcal{A} z(t) + \mathcal{B} f(t), \quad t\in (0,T), \qquad z(0) = z_{0},
	\end{equation}
	where $z(t) = (\rho(t,\cdot), u(t, \cdot), S(t,\cdot))^{\top},$ $z_{0} =(\rho_{0}, u_{0}, S_0)^{\top}\in\mathcal{Z}_m$ , $\left( \mathcal{A},\mathcal{D}\left( \mathcal{A};\mathcal{Z}_m\right) \right) $ is as defined in \eqref{def-A} and the control operator $\mathcal{B} \in \mathcal{L}(L^2(0,\pi);{\mathcal{Z}_m})$ is defined by
	\begin{equation}\label{defn-Bf(t)}
		\mathcal{B}f=\begin{pmatrix}
			f\\0\\0
		\end{pmatrix}, \quad \forall\, f\in L^2(0,\pi).
	\end{equation}

For each $n\in  \mathbb{N}\cup\{0\}$, since $\mathbf{Z}_n$ is the eigenspace of $\mathcal{A}$ with finite dimension, $\mathbf{Z}_n$ is a subspace of $\mathcal{D}(\mathcal A ; \mathcal{Z}_m)$ and the following result holds. 
	\begin{lem}\label{lem-An}
		For each $n\in \mathbb{N}\cup\{0\}$, the space $\mathbf{Z}_n$ defined in \eqref{def-Zn} is invariant under $\mathcal{A}$ and
		\begin{equation}
			\pi_n \mathcal{A}=\pi_n\mathcal{A}\pi_n=\mathcal{A}\pi_n.
		\end{equation}
		For each $n\in \mathbb{N}$, the matrix representation of $\pi_n\mathcal{A}|_{\mathbf{Z}_n}$ in the basis $\left\lbrace \xi_{n,l},\xi_{n,2},\xi_{n,3}\right\rbrace $ is $\mathcal{A}_n$ defined by
		\begin{equation}
			\mathcal{A}_n=\begin{pmatrix}
				\lambda_n^1&0&0\\
				0&\lambda_n^2&0\\
				0&0&\lambda_n^3
			\end{pmatrix},
		\end{equation}
		where $\{\lambda_n^\ell \mid \ell=1,2,3\}$ are the eigenvalues of $\mathcal{A}$ as mentioned in \cref{spectrumofA}. Also for $n=0$, $\mathcal{A}_0=0$.
		
For all $n\in \mathbb{N}\cup\{0\}$, the adjoint of $\pi_n \mathcal{A}|_{\mathbf{Z}_n}$ is defined by $(\mathcal{A}^*\pi^*_n)|_{\mathbf{Z}^*_n}$ with the matrix representation 
$\mathcal{A}^*_n$ in the basis of $\mathbf{Z}^*_n$. 

Furthermore, there exists a positive constant $C$ independent of $n$ such that 
\begin{equation}\label{equniformfinsemigr}
 \|e^{t \pi_n \mathcal{A}}\|_{\mathcal{L}(\mathbf{Z}_n)}\le C, \quad \|e^{t (\pi_n \mathcal{A})^*}\|_{\mathcal{L}(\mathbf{Z}^*_n)}\le C, \quad \forall\, n\in \mathbb{N}\cup \{0\}.
 \end{equation}
	\end{lem}

Now, we introduce the finite dimensional subspaces $E_n$ of $L^2(0,\pi)$ defined by
	\begin{equation}\label{def-E_n}
		E_0=\text{ span }\left\lbrace \frac{1}{\sqrt{\pi}}\right\rbrace,\quad E_n=\text{ span }\left\lbrace \sqrt{\frac{2}{\pi}}\cos (nx);\: x\in (0,\pi)\right\rbrace,\quad n\in\mathbb{N}. 
	\end{equation}
	Then $L^2(0,\pi)=\oplus_{n=0}^{\infty}E_n$ and any $g\in L^2\left(0,T;L^2(0,\pi) \right) $ can be uniquely expressed in the form
	\begin{equation}\label{f=f_n}
		g=\sum\limits_{n=0}^{\infty}g_n,\text{ where }g_n\in E_n, \quad \text{with}\quad 
		\left\| g\right\|^2_{L^2(0,\pi)}= \sum\limits_{n=0}^{\infty}\left\| g_n\right\|^2_{L^2(0,\pi)}.
	\end{equation}

	\begin{lem}\label{lem-Bn}
	Let $g\in L^2(0,\pi) $ be decomposed in the form \eqref{f=f_n}. Then
		\begin{equation}\label{pi_nf=pi_nf_n}
			\pi_n\mathcal{B}g= \pi_n\mathcal{B}g_n,\text{ for all }n\geq 0,
		\end{equation}
and $\pi_n\mathcal{B}|_{E_n}\in \mathcal{L}(E_n, \mathbf{Z}_n)$. 		
For any $n\in \mathbb{N}$, the matrix representation of $\pi_n \mathcal{B}|_{E_n}$ in the basis $\left\lbrace \sqrt{\frac{2}{\pi}}\cos (nx)\right\rbrace $ of $E_n$ and $\left\lbrace \xi_{n,1},\xi_{n,2}, \xi_{n,3}\right\rbrace $ of $\mathbf{Z}_n$ is
		\begin{equation}\label{B_n}
			\mathcal{B}_n= b\sqrt{\frac{\pi}{2}}\begin{pmatrix}
				\frac{1}{{\bar\psi_{n,1}}}\\\frac{1}{{\bar\psi_{n,2}}}\\\frac{1}{{\bar\psi_{n,3}}}
			\end{pmatrix},
		\end{equation}
where $\psi_{n,\ell}$, for $n\in \mathbb{N}$ and $\ell=1,2,3$ is defined in \eqref{equ-psi_nl}. 
For $n=0$, $\mathcal{B}_0=\sqrt{b}$.	

For all, $n\in \mathbb{N}\cup \{0\}$, the adjoint of $\pi_n\mathcal{B}|_{E_n}$ is defined by $\mathcal{B}^*\pi^*_n|_{\mathbf{Z}_n^*}\in \mathcal{L}(\mathbf{Z}_n^*, E_n)$ with the matrix representation $\mathcal{B}^*_n$ in the basis  of $\mathbf{Z}^*_n$ and $E_n$.

Furthermore, there exists a positive constant $C_B>0$ such that 
\begin{equation}\label{equniformbddcntrl}
\|\pi_n\mathcal{B}\|_{\mathcal{L}(E_n, \mathbf{Z}_n)}\le C_B, \quad \|(\pi_n\mathcal{B})^*\|_{\mathcal{L}(\mathbf{Z}_n^*, E_n)}\le C_B, \forall\, n\in \mathbb{N}\cup\{0\}.
\end{equation}
\end{lem}
	\begin{proof}
Since $\mathcal{B}g\in\mathcal{Z}_m$, from \eqref{def-pi}, we have
\begin{equation}\label{eqexx}
\pi_n\mathcal{B}g = \sum_{l=1}^{3}\left\langle \begin{pmatrix}
				g\\0\\0
\end{pmatrix}, \xi_{n,l}^*\right\rangle_{\mathcal{Z}_m}\xi_{n,l},  \quad  \forall\, n\in \mathbb{N}, \quad 
\pi_0\mathcal{B}g = \left\langle \begin{pmatrix}g\\0\\0 \end{pmatrix}, \xi_{0}^*\right\rangle_{\mathcal{Z}_m}\xi_{0}. 
\end{equation}
Now, using \eqref{f=f_n} and the fact that 
\begin{equation*}
			\left\langle \begin{pmatrix}
				g_k\\0\\0
			\end{pmatrix}, \xi_{n,l}^*\right\rangle_{\mathcal{Z}_m}= 0\text{ for }k\neq n,\quad \left\langle \begin{pmatrix}
				g_k\\0\\0
			\end{pmatrix}, \xi_{0}^*\right\rangle_{\mathcal{Z}_m}= 0\text{ for }k\neq 0,
		\end{equation*}
from \eqref{eqexx}, \eqref{pi_nf=pi_nf_n} follows.

Now note that $\pi_n\mathcal{B}|_{E_n}\in \mathcal{L}(E_n, \mathbf{Z}_n)$ and for all $x\in (0,\pi)$, $g_n(x)=\sqrt{\frac{2}{\pi}}\cos (nx)$, for $n\in \mathbb{N}$ and 
$g_0(x)=\sqrt{\frac{1}{\pi}}$, we get 
$$
\pi_n\mathcal{B}g_n =b\sqrt{\frac{\pi}{2}}\sum\limits_{l=1}^{3}\frac{1}{{\bar\psi_{n,l}}}\xi_{n,l}, \text{ for all} \, n\geq 1, 
	\quad 
\pi_0\mathcal{B}g_0 = \sqrt{b}\xi_0.
$$
Thus, the representation $\mathcal{B}_n$ for all $n\in \mathbb{N}\cup \{0\}$ holds.  Using \eqref{cgs-theta_nl} and above representation, \eqref{equniformbddcntrl} can be proved. 
\end{proof}

For each $n\in \mathbb{N}\cup\{0\}$, projecting \eqref{opp-eqn} on $\mathbf{Z}_n$ and considering control from $E_n$,  we obtain the finite dimensional system 
	\begin{equation} \label{opp-eqn_proj}
		\dot{z}_n(t) = \pi_n\mathcal{A} z_n(t) + \pi_n\mathcal{B}f_n(t), \quad t\in (0,T), \qquad z_n(0) =z_{0,n},
	\end{equation}
	where $\dot{z}_n(t)=\pi_n\dot{z}(t), \, z_n(t)=\pi_n z(t), \, z_{0,n}=\pi_n z_0$.

\section{ Null controllability}\label{null controllability}\label{secnull}
	The goal of this section is to prove \cref{thm_pos}. 
For this, we first prove the null controllability of each finite dimensional projected system \eqref{opp-eqn_proj} and using this, the null controllability of the system \eqref{opp-eqn} is obtained.

Let us recall that \eqref{opp-eqn_proj} is posed on $\mathbf{Z}_n$ with control from $E_n$, where  $\mathbf{Z}_n$ and $E_n$ are finite dimensional spaces defined in \eqref{def-Zn} and \eqref{def-E_n} respectively. Recall from \cref{lem-An} and \cref{lem-Bn}, the matrix representation of $\pi_n\mathcal{A}|_{\mathbf{Z}_n}$ and $\pi_n\mathcal{B}|_{E_n}$ is given by $\mathcal{A}_n$ and $\mathcal{B}_n$, respectively. To determine the controllability of \eqref{opp-eqn_proj}, we check the \textquotedblleft Hautus Lemma\textquotedblright (see \cite[Theorem 2.5, Section 2, Chapter 1, Part I]{bensoussan2007representation}) for the pair $(\mathcal{A}_n, \mathcal{B}_n)$. 
	\begin{thm}\label{th_finitecontrol}
	For each $n\in\mathbb{N}\cup\{0\}$, the finite dimensional system \eqref{opp-eqn_proj} is controllable at any given $T>0$.
	For each $n\in\mathbb{N}\cup\{0\}$, the control $f_n\in L^2(0,T; E_n)$ which brings the solution of \eqref{opp-eqn_proj}  at rest at a given time $T>0$ is given by 
	\begin{equation}\label{equ-mincontrol}
			f_n(t)= - \mathcal{B}^*_ne^{(T-t)\mathcal{A}_n^*}W^{-1}_{n,T}e^{T\mathcal{A}_n} z_{0,n},\qquad t\in[0,T],
		\end{equation}
		where $W_{n,T}\in \mathcal{L}(\mathbf{Z}^*_n, \mathbf{Z}_n)$ is the controllability operator given by
 \begin{equation}\label{eqgram}
 W_{n,T}= \int_{0}^{T}e^{t\pi_n\mathcal{A}}\Big(\pi_n\mathcal{B}\Big)\Big(\pi_n\mathcal{B}\Big)^*e^{t(\pi_n\mathcal{A})^*}\:\rd t. 
\end{equation}
	\end{thm}
\begin{proof}
	Let $n\in \mathbb{N}$ be fixed. We prove this result showing $(iii)$ of \cite[Theorem 2.5, Section 2, Chapter 1, Part I]{bensoussan2007representation}). We show that for $n\in \N$
	the $3\times 3$ matrix $\mathcal{A}_n$ and $3\times 1$ matrix $\mathcal{B}_n$ satisfies 
	\begin{align*}
		\text{Rank}\left[\lambda I-\mathcal{A}_n \;;\; \mathcal{B}_n \right]=3,\text{ for all $\lambda\in\mathbb{C}$ that are eigenvalues of }\mathcal{A}_n, 
	\end{align*}
	and for $n=0$, the row vector $\left[\lambda_0 I-\mathcal{A}_0 \;;\; \mathcal{B}_0 \right]$ has rank $1$. 
	
	From direct calculation, we have 
	\begin{equation}
		\left[\lambda^\ell I-\mathcal{A}_n \;;\; \mathcal{B}_n \right]=\begin{bmatrix}
			\lambda^\ell-\lambda_n^1&0&0&b\sqrt{\pi/2}\frac{1}{\bar\psi_{n,1}}\\
			0&\lambda^\ell-\lambda_n^2&0&b\sqrt{\pi/2}\frac{1}{\bar\psi_{n,2}}\\
			0&0&\lambda^\ell-\lambda_n^3&b\sqrt{\pi/2}\frac{1}{\bar\psi_{n,3}}
		\end{bmatrix},
	\end{equation}
 where the eigenvalues of $\mathcal{A}_n$ are $\lambda_n^\ell, \: \ell=1,2,3.$ In particular, for $\lambda^\ell=\lambda_n^1$, note that the $3\times 4$ matrix has $3$ independent columns and thus Rank$\left[\lambda_n^1 I-\mathcal{A}_n \;;\; \mathcal{B}_n \right]= 3$. Similarly we can see that Rank$\left[\lambda_n^\ell I-\mathcal{A}_n \;;\; \mathcal{B}_n \right]= 3$, for $\ell=2,3$. 
For $n=0$, $\left[\lambda_0 I-\mathcal{A}_0 \;;\; \mathcal{B}_0 \right]= [0, \sqrt{b}]$, which implies Rank$\left[\lambda_0 I-\mathcal{A}_n \;;\; \mathcal{B}_n \right]= 1$. 
Therefore for each $n\in\mathbb{N}\cup \{0\}$, the system \eqref{opp-eqn_proj} is controllable at any time $T$.

Since $(\pi_n\mathcal{A}, \pi_n\mathcal{B})$ is controllable, from \cite[Part I, Chapter 1, Proposition 1.1]{zabczyk2008mathematical}, it follows that the inverse of the controllability operator $W_{n,T}^{-1}$ exists  in $\mathcal{L}(\mathbf{Z}_n, \mathbf{Z}^*_n)$ and the control $f_n$ given in  \eqref{equ-mincontrol} is well-defined and it gives the null controllability of \eqref{opp-eqn_proj} at time $T$. 
\end{proof}

In the next lemma, we derive the uniform (in $n$) estimate of the controllability operator $W_{n,T}$ and hence the control $f_n$. 

\begin{lem}\label{lem-cntrlfin-est}
For each $n\in \mathbb{N}\cup \{0\}$, let $W_{n,T}$ and $f_n\in L^2(0,T; E_n)$ be as defined in \eqref{eqgram} and \eqref{equ-mincontrol}, respectively. 
Then, there exist positive constants $C_1$ and $C_2$ independent of $n$, such that 
\begin{equation}\label{bddWnT}
 C_1\le \|W_{n,T}\|_{\mathcal{L}(\mathbf{Z}^*_n, \mathbf{Z}_n)}\le C_2, \quad \forall\, n\in \mathbb{N}\cup\{0\},
 \end{equation}
and 
\begin{equation}\label{estimatef_n}
	\int_{0}^{T}\left\| f_n(t)\right\|^2_{L^2(0,\pi)}\:\rd t \leq C \left\|z_{0,n} \right\|^2_{\mathcal{Z}},
		\end{equation}
for some positive constant $C$ independent of $n$.
\end{lem}
\begin{proof}
Let us observe that the matrix representation of the controllability operator $W_{n,T}\in \mathcal{L}(\mathbf{Z}^*_n, \mathbf{Z}_n)$
(still denoted by same notation $W_{n,T}$) in the basis $\{\xi^*_{n,\ell}\mid \ell=1,2,3 \}$ of $\mathbf{Z}^*_n$ and $\{\xi_{n,\ell} \mid \ell=1,2,3 \}$ of $\mathbf{Z}_n$ is given by 
\begin{equation}\label{matWnT}
				W_{n,T}=b^2\pi/2\begin{pmatrix}
					\frac{e^{2T\lambda_n^1}-1}{2\lambda_n^1\left| \psi_{n,1}\right|^2 } &\frac{e^{T\left(\lambda_n^1+\bar{\lambda}_n^2 \right) }-1}{\left(\lambda_n^1+\bar{\lambda}_n^2 \right) \bar{\psi}_{n,1}\psi_{n,2} } & \frac{e^{T\left(\lambda_n^1+\bar{\lambda}_n^3 \right) }-1}{\left(\lambda_n^1+\bar{\lambda}_n^3 \right) \bar{\psi}_{n,1}\psi_{n,3} } \\
					\frac{e^{T\left(\lambda_n^1+{\lambda}_n^2 \right) }-1}{\left(\lambda_n^1+{\lambda}_n^2 \right) \bar{\psi}_{n,2}\psi_{n,1} } &\frac{e^{2T\text{ Re}\lambda_n^2}-1}{2\text{ Re}\lambda_n^2\left| \psi_{n,2}\right|^2 }&\frac{e^{T\left(\lambda_n^2+\bar{\lambda}_n^3 \right) }-1}{\left(\lambda_n^2+\bar{\lambda}_n^3 \right) \bar{\psi}_{n,2}\psi_{n,3} } \\
					\frac{e^{T\left(\lambda_n^1+{\lambda}_n^3 \right) }-1}{\left(\lambda_n^1+{\lambda}_n^3 \right) \bar{\psi}_{n,3}\psi_{n,1} } &\frac{e^{T\left(\lambda_n^3+\bar{\lambda}_n^2 \right) }-1}{\left(\lambda_n^3+\bar{\lambda}_n^2 \right) \bar{\psi}_{n,3}\psi_{n,2} } &\frac{e^{2T\text{ Re}\lambda_n^3}-1}{2\text{ Re}\lambda_n^3\left| \psi_{n,3}\right|^2 }
				\end{pmatrix},
			\end{equation}
where $W_{n,T}^{ij}$, for $i, j\in \{1,2,3\}$, is the $(ij)$th element of the above matrix \eqref{matWnT} satisfying 
			\begin{align}
				&\left| W_{n,T}^{11}\right|\rightarrow \frac{\mu\left( 1-e^{2Tw_0}\right) }{\pi b^2\rho^2_s},\quad  \left| W_{n,T}^{22}\right|, \left| W_{n,T}^{33}\right|\rightarrow \frac{\kappa^2\rho_s^2\left( 1-e^{-T\left( w_0+1/\kappa\right) }\right) }{\pi \mu},\label{estWnii}\\
				&\text{ and }\left| W_{n,T}^{ij}\right|\rightarrow 0,\quad\text{ for }i\neq j,\: i,j\in \{1,2,3\},\label{estWnij} \quad\text{as }n\rightarrow \infty,
			\end{align}
due to \cref{spectrumofA}$(c)$ and \eqref{cgs-theta_nl}. 
Noting that 
		$$ M_1 \Big(\sum_{i=1}^3\sum_{j=1}^3|W_{n,T}^{ij}|^2\Big)^{\frac{1}{2}} \le \left\|W_{n,T} \right\|_{\mathcal{L}\left(\mathbf{Z}^*_n,\mathbf{Z}_n \right)} \le M_2\Big(\sum_{i=1}^3\sum_{j=1}^3|W_{n,T}^{ij}|^2\Big)^{\frac{1}{2}}, \quad \forall\, n\in \mathbb{N},$$
for some positive constants $M_1$ and $M_2$ independent of $n$, from \eqref{estWnii}-\eqref{estWnij}, \eqref{bddWnT} follows. 

From \eqref{equniformfinsemigr}, \eqref{equniformbddcntrl}, \eqref{equ-mincontrol} and \eqref{bddWnT}, \eqref{estimatef_n} follows. 
\end{proof}

Now we are in a position to prove \cref{thm_pos}.\\
\textbf{Proof of \cref{thm_pos} : }
To prove \cref{thm_pos}, we show that for any $T>0$ and for any $z_0\in \mathcal{Z}_m$, there exists a control $f\in L^2(0,T; L^2(0,\pi))$ such that $z$, the solution of \eqref{opp-eqn} satisfies $z(T)=0$ in $\mathcal{Z}_m$. 

Let $z_{0} \in  {\mathcal{Z}_m}$. Then for each $n\in \mathbb{N}\cup \{0\}$, the control $f_n\in L^2(0,T; E_n)$ defined in \eqref{equ-mincontrol} brings the solution of the finite dimensional system \eqref{opp-eqn_proj} with initial condition $z_{0,n}$ at rest at time $T$ and furthermore, the control $f_n$ satisfies \eqref{estimatef_n}. 
	
Let us set $f=\sum\limits_{n=0}^{\infty}f_n$. Using \eqref{estimatef_n} and \eqref{equ-est-z_0}, we get $f\in L^2(0,T; L^2(0,\pi))$ with 
	\begin{align*}
		\sum_{n=0}^{\infty}\left\|  f_n\right\|^2_{L^2(0,T;L^2(0,\pi))} \leq C\sum_{n=0}^{\infty}\left\|z_{0,n} \right\|^2_{\mathcal{Z}}
		\le C \left(\left| \left\langle z_0, \xi^*_{0} \right\rangle\right|^2+ \sum_{n=1}^{\infty}\sum_{l=1}^{3}\left| \left\langle z_0, \xi^*_{n,l} \right\rangle\right|^2 \right)
		\leq C\left\|z_0 \right\|_{\mathcal{Z}_m}^2,
	\end{align*}
	for some generic positive constant $C$.
With this control $f$ and the initial condition $z_0$, \eqref{opp-eqn} has a unique solution $z\in C([0,T]; \mathcal{Z}_m)$ and $z(t)$ can be decomposed as 
$z(t)=\sum_{n=0}^\infty z_n(t)$, where $z_n(t)\in \mathbf{Z}_n$ satisfies \eqref{opp-eqn_proj}. Now, from \cref{th_finitecontrol}, it follows that $z(T)=0$ in $\mathcal{Z}_m$. 
\qed

	\begin{rem}
		Let us consider the
		linearized Navier-Stokes system of a viscous, compressible, isothermal barotropic fluid in a bounded domain $(0,\pi)$ around a constatnt steady state $(Q_0, 0)$ with $Q_0>0$ as in \cite{chowdhury2012controllability},
		\begin{equation}\label{CNSE1}
			\begin{dcases}
				\partial_t\rho+{Q_0}\partial_xu= f,&\text{ in }(0, T)\times (0, \pi),\\
				\partial_tu-\frac{\mu}{Q_0}\partial_{xx}u+a\gamma{Q_0}^{\gamma-2}\partial_x\rho=0,&\text{ in }(0, T)\times (0, \pi),\\
				\rho(0)=\rho_0,\quad u(0)=u_0, & \text{ in } (0,\pi),\\
				u(t,0)=0,\quad u(t,\pi)=0,& \text{ in } (0, T).\\
			\end{dcases}
		\end{equation}
		Then using a similar approach as for the proof of \cref{thm_pos}, we can show the following result for the system \eqref{CNSE1}.
		
		$\bullet$	Let us assume $\left( \rho_0, u_0\right)\in L^{2}(0,\pi) \times L^{2}(0,\pi)$. Then for any $T>0$, there exists a control $f\in L^2\left(0,T ; L^2(0,\pi) \right) $ acting everywhere in the density equation, such that $(\rho,u)$, the corresponding solution to \eqref{CNSE1}, belongs
		to $C\left( [0,T];  L^{2}(0,\pi) \times L^{2}(0,\pi)\right) $ and satisfies
		\begin{equation}
			\rho(T,x)= u(T,x)=0\text{ for all }x\in (0,\pi).
		\end{equation}
	\end{rem}

	%%%%%%%%%%%%%%%%%%%%%%%%%%%%%%%%%%%%%%%%%%%%%%%%%%%%%%%%%%%%%%%%%%%%%%%%%%%%%%%%%%%%%%%%%%%%%%%%%%%%

\section{ Unique continuation and approximate controllability}\label{Approximate controllability}\label{secapprox}
In this section we study the approximate controllability of the system \eqref{eq2}-\eqref{bdd-ini}. The approximate controllability of a linear system is equivalent to showing the unique continuation of the solution of the corresponding adjoint system.	

We recall that ${\mathcal Z_{m}}= L^2(0, \pi)\times L^2(0, \pi)\times L_m^2(0, \pi)$.
Let $f_2=0=f_3$ in \eqref{eq2}. Then the control operator $\mathcal{B} \in \mathcal{L}(L^2(0,\pi);{\mathcal Z_{m}})$ is given by 
\begin{equation}\label{equcontr}
	\mathcal{B} f_1  = \left( \mathbbm{1_{\mathcal{O}_1}} f_1, 0,0 \right)^{\top}, \qquad f_1 \in L^2(0,\pi),
\end{equation}
and $\mathcal{B}^*$, the adjoint of $\mathcal{B}$, belongs to $\mathcal{L}({\mathcal Z_{m}};L^2(0,\pi))$ given by 
\begin{equation}\label{def_B^*}
	\mathcal{B}^*\Phi=b\mathbbm{1_{\mathcal{O}_1}}\phi_1,\quad\Phi=( \phi_1, \phi_2, \phi_3)^\top \in{\mathcal Z_{m}}.
\end{equation}

Let us recall that  $(\mathcal{A}^*, \mathcal{D}(\mathcal{A}^*; \mathcal{Z}_m))$ defined in \eqref{dom-A*}-\eqref{op-A*} generates the $C^{0}$-semigroup $\mathbb{T}^*$ on 
$\mathcal{Z}_m$. It is known that the approximate controllability of a pair $(\mathcal{A}, \mathcal{B})$ at time $T$ is equivalent to the unique continuation property that if 
$(\sigma^0, v^0, \tilde{S}^0)\in \mathcal{Z}_m$ and $\mathcal{B}^*\mathbb{T}^*_{t}(\sigma^0, v^0, \tilde{S}^0)^\top=0$ for almost all $t\in [0,T]$, then 
$(\sigma^0, v^0, \tilde{S}^0)^\top=(0,0,0)^\top$. 
For details see \cite[Chapter 2, Theorem 2.43]{coron2007control}.

In view of \eqref{def_B^*} and the above equivalence conditions, \cref{main_approximate} is equivalent to the following result.
\begin{thm}\label{thm_appequivucp}
	Let $T >\frac{4\pi}{\sqrt{b\rho_s+\frac{\mu}{\kappa\rho_s}}}$. For any $(\sigma^{0}, v^{0}, \tilde{S}^0 ) \in \mathcal{Z}_m,$ let
	\begin{equation}\label{eq-exp-adj-semi}
		(\sigma(t), v(t), \tilde{S}(t))^\top = \mathbb{T}_{t}^{*} (\sigma^{0}, v^{0}, \tilde{S}^0)^\top \qquad (t \geqslant 0), \quad \forall\, t\in (0,T),
	\end{equation}
where $\mathbb{T}^*$  is the $C^{0}$-semigroup  generated by $(\mathcal{A}^*, \mathcal{D}(\mathcal{A}^*; \mathcal{Z}_m))$ defined in \eqref{dom-A*}-\eqref{op-A*}, on 
$\mathcal{Z}_m$.
If $$\sigma=0 \quad \mathrm{in}\quad (0,T)\times \mathcal{O}_1,$$ 
then $\left(\sigma^0,v^0,\tilde S^0 \right)^{\top}=(0,0,0)^{\top}$ in $(0,\pi)$.
\end{thm}

We prove the above theorem using the series expansion of $(\sigma^{0}, v^{0}, \tilde{S}^0 ) \in \mathcal{Z}_m$. Recall that, for any 
$(\sigma^{0}, v^{0}, \tilde{S}^0 ) \in \mathcal{Z}_m,$ from \cref{basis_representation}, we have
\begin{equation}\label{est_alphanl}
	\begin{pmatrix}
		\sigma^{0}\\v^{0}\\\tilde{S}^{0}
	\end{pmatrix}=\sum\limits_{n=1}^{\infty}\sum\limits_{l=1}^{3}\alpha_{n,l} \xi^*_{n,l} + \alpha_0\xi_0^*, \quad \mathrm{with} \quad 
	\sum\limits_{n=1}^{\infty}\sum\limits_{l=1}^{3}\left| \alpha_{n,l}\right|^2 + \left|\alpha_0 \right|^2  < \infty .
\end{equation}
Then $(\sigma(t), v(t), \tilde{S}(t))^\top$ defined in \eqref{eq-exp-adj-semi} can be written as
\begin{equation}\label{repressoln}
	\begin{pmatrix}
		\sigma\\v\\\tilde{S}
	\end{pmatrix}(t,x)=\sum\limits_{n=1}^{\infty}\sum\limits_{l=1}^{3}\alpha_{n,l}e^{\lambda_n^lt} \xi^*_{n,l}(x) + \alpha_0e^{\lambda_0 t}\xi_0^*(x).
\end{equation}
In particular, using the expression of $\xi^*_0$ and $\xi^*_{n,l}$ ( eq. \eqref{equ-xi_0^*}} and \eqref{equ-xi_n^*}), from \eqref{repressoln} we get
\begin{equation}\label{expansion_v}
	\sigma(t,x)=\sum\limits_{n=1}^{\infty}\sum\limits_{l=1}^{3}\frac{\alpha_{n,l}}{\psi_{n,l}}e^{\lambda_n^lt}\cos(nx) + \frac{\alpha_0}{\sqrt{b\pi}}, \quad \forall\, t\in (0,T), \quad x\in (0,\pi).
\end{equation}

If $\sigma=0$ in $(0,T)\times \mathcal{O}_1$, we show that 
\begin{equation}\label{alphanl0}
	\alpha_0=0\text{ and }\alpha_{n,l}=0 \text{ for all } n\in\mathbb{N} \text{ and } l=1,2,3,
\end{equation}
and thus from \eqref{est_alphanl}, \cref{thm_appequivucp} follows.

To show \eqref{alphanl0}, we need an analytic extension of the mapping $t\mapsto \sigma(t,\cdot)\in L^2(0,\pi)$ from $(0,T)$  to $(0,\infty)$. 
To do it, from \cref{spectrumofA}, first recall that for all $n\in \mathbb{N}$, $\lambda_n^1$ is a negative real number satisfying $\displaystyle \lim_{n\rightarrow \infty} \lambda^1_n=\omega_0<0$. In view of \eqref{est-lambda_n2R}, a large enough $N_0\in \mathbb{N}$ can be chosen such that for all $n\ge N_0$, $\lambda_n^2$, $\lambda_n^3$ are complex conjugate satisfying 
\begin{equation}\label{epsilondelta}
\begin{array}{l}
	\lambda_n^3 =\overline{\lambda_n^2}, \quad 
	\lambda_n^2=	-\frac{1}{2}\left( w_0+\frac{1}{\kappa}\right)+\epsilon_n + i\left(  n\sqrt{(b\rho_s+\frac{\mu}{\kappa\rho_s})}+ \delta_n\right), \quad \forall\, n\ge N_0,\\
	\text{where}\quad \epsilon_n\rightarrow 0\text{ and } \delta_n\rightarrow  0,\text{ as }n\rightarrow \infty, \quad |\delta_n|\le \frac{\sqrt{(b\rho_s+\frac{\mu}{\kappa\rho_s})}}{2}, \quad \forall\, n\ge N_0. 
\end{array}
\end{equation}
From \cref{spectrumofA} it follows that for $1\le n\le N_0-1$, both $\lambda_n^2$, $\lambda_n^3$ can be negative real numbers or both can be complex with negative real part.

Now we split $\sigma$ in the following way
\begin{equation}\label{split_v}
	\sigma(t,x)=\sigma_h(t,x)+\sigma_p(t,x), \quad \forall\, t\in (0,T), \quad x\in (0,\pi),
\end{equation}
where for all $t\in (0,T)$ and $x\in (0,\pi)$, 
\begin{align}
	\sigma_h(t,x)=&\sum\limits_{n=M}^{\infty}\sum\limits_{l=2}^{3}\frac{\alpha_{n,l}}{\psi_{n,l}}e^{{\lambda_n^l}t}\cos(nx),\label{def_v}\\
	\sigma_p(t,x)=&\sum\limits_{n=1}^{M-1}\sum\limits_{l=1}^{3}\frac{\alpha_{n,l}}{\psi_{n,l}}e^{\lambda_n^lt}\cos(nx) + \frac{\alpha_0}{\sqrt{b\pi}}
	+\sum\limits_{n=M}^{\infty}\frac{\alpha_{n,1}}{\psi_{n,1}}e^{\lambda_n^1t}\cos(nx),\label{def_v2}
\end{align}
where $M>N_0$ is chosen suitably large enough. 

\begin{lem}\label{lemparabolicpart}
Let $\sigma_p$ be defined as \eqref{def_v2}. We set the extension of $\sigma_p$ for all $z\in \mathbb{C}, \, \mathrm{Re}\, z>0$ with the representation formula
for all $z\in \C, \, \mathrm{Re}\, z>0$,
\begin{equation}\label{extenv2}
\widetilde{\sigma}_p(z,x)=\sum\limits_{n=1}^{M-1}\sum\limits_{l=1}^{3}\frac{\alpha_{n,l}}{\psi_{n,l}}e^{\lambda_n^l z}\cos(nx) + \frac{\alpha_0}{\sqrt{b\pi}}
	+\sum\limits_{n=M}^{\infty}\frac{\alpha_{n,1}}{\psi_{n,1}}e^{\lambda_n^1 z}\cos(nx), \quad x\in (0,\pi).
\end{equation}	
Then, the mapping $z\mapsto \widetilde{\sigma}_p(z, \cdot)\in L^2(0,\pi)$ is analytic on $\{z\in \mathbb{C}\mid \mathrm{Re}\, z>0\}$ and 
$\widetilde{\sigma}_p=\sigma_p$ on $(0,T)\times (0,\pi)$. 

Furthermore, there exist positive constants $C_1$ and $C_2$, independent of $t$ and $r$ such that 
\begin{equation}\label{bddv_2}
	\left\| \frac{d^r \sigma_p}{d t^r}(t,\cdot)\right\|_{L^2(0,\pi)}\leq C_1(C_2)^r\quad \text{for all}\quad t\in (0,T). 
\end{equation}
\end{lem} 
\begin{proof}
Using $(a), (b), (c)$ of \cref{spectrumofA}, \eqref{cgs-theta_nl} and \eqref{est_alphanl}, we have the series in \eqref{extenv2} is absolutely convergent in $L^2(0,\pi)$-norm for all $z\in \mathbb{C}$ with $\mathrm{Re}\, z>0$ and hence $\widetilde{\sigma}_p$ in \eqref{extenv2} is well-defined. In fact, it can be shown that the series  is uniformly convergent in $L^2(0,\pi)$-norm on any compact subset of $\{z\in \mathbb{C}\mid \mathrm{Re}\, z>0\}$.
Using this result and the fact that $e^{\lambda_n^1 z} \left(n\ge 1\right)$ and $e^{\lambda_n^\ell z} \left( l=2,3, 1\leq n \leq M-1\right) $ are entire functions, 
we conclude that the mapping $z\mapsto \widetilde{\sigma}_p(z,\cdot)\in L^2(0,\pi)$ is analytic on $\{z\in \mathbb{C}\mid \mathrm{Re}\, z>0\}$. From representation, it follows 
$\widetilde{\sigma}_p=\sigma_p$ on $(0,T)\times (0,\pi)$. 

From \eqref{def_v2}, by a direct estimate, \eqref{bddv_2} follows. 
\end{proof}

Before showing the extension of $\sigma_h$ for all $t\in (0,\infty)$, we need the following estimate. 

\begin{lem}\label{Ingham}
Let $T> \frac{4\pi}{\sqrt{(b\rho_s+\frac{\mu}{\kappa\rho_s})}}$. Let $\sigma_h$ and $\sigma_p$ be as defined in \eqref{def_v} and \eqref{def_v2} such that $\sigma^h=-\sigma_p$ on $(0,T)\times \mathcal{O}_1$, where $\sigma_p$ satisfies \eqref{bddv_2}. Then there exists a positive constant $C$, depending on $T$, such that 
for all $r\in \mathbb{N}\cup \{0\}$,
\begin{equation}\label{bddsigma1}
	\left\|\frac{d^r \sigma_h}{dt^r}(t,\cdot) \right\|_{L^2(0,\pi)} \leq C \left\|\frac{d^r \sigma_h}{dt^r} \right\|_{L^2\left( 0,T; L^2(\mathcal{O}_1)\right) }\quad \text{for all} \quad t\in (0,T).
\end{equation}
\end{lem}
\begin{proof}
We prove this lemma using the series expansion of $\sigma_h$ given in \eqref{def_v}.
Recall for all $n\ge M$, $\lambda_n^2$ and $\lambda_n^3$ satisfy \eqref{epsilondelta}. Setting 
\begin{align}\label{defbetan}
	\beta_n(x)=\begin{cases}
		\frac{\alpha_{n,2}}{\psi_{n,2}}\cos(nx) \: n\geq M,\\
		\frac{\alpha_{-n,3}}{\psi_{-n,3}}\cos(-nx), \: n\leq -M,
	\end{cases}
	\forall\, x\in (0,\pi),
\end{align}
and 
\begin{align}\label{defnmun}
	\mu_n=\begin{cases}
		-i\lambda^2_n,  \: n\geq M\\
		-i\lambda^3_{-n},  \: n\leq -M,
	\end{cases}
\end{align}
from \eqref{def_v}  we have
\begin{equation}\label{sigma1beta}
	\sigma_h(t,x)=\sum\limits_{\left| n\right| \geq M}\beta_{n}(x)e^{i\mu_n t},\quad t\in (0,T),\: x\in (0,\pi),
\end{equation}
where for each $x\in (0,\pi)$, $\sum_{|n|\ge M}|\beta_n(x)|^2<\infty$ due to \eqref{est_alphanl} and \eqref{cgs-theta_nl}. From \eqref{epsilondelta}, the gap condition holds:
\begin{equation}\label{gapcondn}
	\mathrm{Re}\,\mu_{n+1}-\mathrm{Re}\,\mu_n\geq \frac{\sqrt{(b\rho_s+\frac{\mu}{\kappa\rho_s})}}{2}:=\gamma > 0,\quad\text{ for }\left| n\right| \geq M.
\end{equation}
Thus, for large enough $M$, using an Ingham type inequality (see \cref{propI-4} in the appendix), we get that for any $T>\frac{4\pi}{\sqrt{(b\rho_s+\frac{\mu}{\kappa\rho_s})}}$, there exists a positive constant $C$ depending on $T$ such that 
\begin{align}
	C\sum\limits_{\left| n\right| \geq M}\left| \beta_n(x)\right|^2\leq \int_{0}^{T}\left|\sum\limits_{\left| n\right| \geq M}\beta_{n}(x)e^{i\mu_n t} \right|^2 \rd t. 
\end{align}
Now integrating the both side over $\mathcal{O}_1$ and noting $\inf\limits_{\left|n \right|\geq M }\int_{\mathcal{O}_1}\cos^2(nx) \rd x >0$, we get 
\begin{align}\label{obsin1}
	\sum\limits_{n\geq M} \sum\limits_{\ell=2}^3\Big|\frac{\alpha_{n,\ell}}{\psi_{n,\ell}}\Big|^2 \leq C \int_{0}^{T}\int_{\mathcal{O}_1}|\sigma_h(t,x)|^2 \rd x \rd t. 
\end{align}
Using the orthogonality of $\{\cos(nx) \mid n\ge M, \quad x\in (0,\pi)\}$ in $L^2(0,\pi)$, the fact that $\mathrm{Re}\, \lambda_n^\ell<0$ for all $n\ge M$ and $\ell=2,3$ and \eqref{cgs-theta_nl}, from \eqref{def_v}, \eqref{bddsigma1} follows for $r=0$. 

For any $r\in \mathbb{N}$, \eqref{bddsigma1} can be proved in an analogous way noting that $\frac{d^r\sigma_h}{dt^r}(t, \cdot)$ exists in $L^2(\mathcal{O}_1)$ with the estimate 
$\Big\|\frac{d^r\sigma_h}{dt^r}(t, \cdot)\Big\|_{L^2(\mathcal{O}_1)}\le C_1C_2^r$ for all $t\in (0,T)$, 
because of the fact $\sigma_h=-\sigma_p$ on $(0,T)\times \mathcal{O}_1$, and \eqref{bddv_2}.
\end{proof}

\begin{lem}\label{lemhyperpart}
Let $\sigma$ be given by \eqref{split_v} where $\sigma_h$ and $\sigma_p$ are defined in \eqref{def_v} and \eqref{def_v2}, respectively. Let $T> \frac{4\pi}{\sqrt{(b\rho_s+\frac{\mu}{\kappa\rho_s})}}$. Let us assume 
$\sigma=0$ on $(0,T)\times \mathcal{O}_1$. Then, there exists a unique analytic function (in $t$) $\widetilde{\sigma}_h(t,\cdot)\in L^2(0,\pi)$, for all $t\in (0,\infty)$ such that $\widetilde{\sigma}_h(t,x)=\sigma_h(t,x)$ for all $t\in (0,T)$ and $x\in (0,\pi)$. 

Moreover,  $\widetilde{\sigma}_h$ has the representation formula
$$ \widetilde{\sigma}_h(t,x)=\sum\limits_{n=M}^{\infty}\sum\limits_{l=2}^{3}\frac{\alpha_{n,l}}{\psi_{n,l}}e^{{\lambda_n^l}t}\cos(nx), \quad \forall\, t\in (0,\infty), \quad x\in (0,\pi).$$
\end{lem}
\begin{proof}
Since $\sigma=0$ on $(0,T)\times \mathcal{O}_1$, using \eqref{split_v}, \eqref{bddv_2} and Lemma \ref{Ingham}, we get 
\begin{equation}\label{boundedv_2}
	\left\| \frac{d^r \sigma_h}{d t^r}(t,\cdot)\right\|_{L^2(0,\pi)}\leq C_3(C_4)^r\quad \text{for }t\in (0,T), \quad \forall\, r\in \mathbb{N}\cup \{0\},
\end{equation}
for some positive constants $C_3$ and $C_4$ depending on $T$.
Using this estimate, it can be proved that the mapping $t\mapsto \sigma_h(t, \cdot)\in L^2(0,\pi)$ is analytic for all $t\in (0,T)$ and its Taylor expansion around any $t \in (0,T)$ is uniformly convergent on $(t-R, t+R)$, for some $R>0$ independent of $t$. 
Thus, note that for any arbitrary $\epsilon>0$, the Taylor expansion of $\sigma_h$ around $t=T-\epsilon$ is defined on $[T-\epsilon, T-\epsilon+R)$. Since $\epsilon>0$ is arbitrary, 
using this Taylor expansion of $\sigma_h$ beyond $T$,  we obtain  $\widetilde{\sigma}_h$, an analytic extension of $\sigma_h$ on $(0, T+R)$. Now repeating this argument, we get $\widetilde{\sigma}_h$, an analytic extension of $\sigma_h$ in $L^2(0,\pi)$ for all $t\in (0,\infty)$.

Since for all $t\in (0,\infty)$, $\widetilde{\sigma}^h(t, \cdot)\in L^2(0,\pi)$, we have 
$$ \widetilde{\sigma}^h(t, x)=\sum_{n=1}^\infty \widetilde{\beta}_n(t) \cos(nx)+\widetilde{\beta}_0, \quad \forall\, t\in (0,\infty), \, x\in (0,\pi),$$
where for each $n\in \mathbb{N}\cup \{0\}$, $t\mapsto \widetilde{\beta}_n(t)$ is analytic for all $t\in (0,\infty)$ because $t\mapsto \widetilde{\sigma}^h(t, \cdot)$ is analytic for all $t\in (0,\infty)$. Since $\widetilde{\sigma}^h=\sigma_h$ on $(0,T)\times (0,\pi)$, from \eqref{def_v}, it follows that 
$$ \widetilde{\beta}_n(t)=\sum\limits_{l=2}^{3}\frac{\alpha_{n,l}}{\psi_{n,l}}e^{{\lambda_n^l}t}, \quad \forall\, n\ge M, \quad
\widetilde{\beta}_n(t)=0, \quad \forall n=0, \cdots, M-1, \quad \forall\, t\in (0,T).$$
Using the uniqueness of analytic continuation \cite[Corollary 1.2.5, Chapter 1]{krantz2002primer}, we get 
$$ \widetilde{\beta}_n(t)=\sum\limits_{l=2}^{3}\frac{\alpha_{n,l}}{\psi_{n,l}}e^{{\lambda_n^l}t}, \quad \forall\, n\ge M, \quad
\widetilde{\beta}_n(t)=0, \quad \forall n=0, \cdots, M-1, \quad \forall\, t\in (0,\infty),$$
and hence the representation formula for $\widetilde{\sigma}_h$ holds. 
\end{proof}

\begin{lem}\label{lemextwhole}
Let $\sigma$ be given by \eqref{split_v} where $\sigma_h$ and $\sigma_p$ are defined in \eqref{def_v} and \eqref{def_v2}, respectively. Let $T> \frac{4\pi}{\sqrt{(b\rho_s+\frac{\mu}{\kappa\rho_s})}}$.  Let us assume 
$\sigma=0$ on $(0,T)\times \mathcal{O}_1$. 
Let $\widetilde{\sigma}_h$ and $\widetilde{\sigma}_p$ be the analytic extension (in $t$) of $\sigma_h$ and $\sigma_p$ on $(0,\infty)\times (0,\pi)$ as obtained in \cref{lemparabolicpart} and \cref{lemhyperpart}. Setting $\widetilde{\sigma}=\widetilde{\sigma}_h+\widetilde{\sigma}_p$, we have 
$$\widetilde{\sigma}=0, \quad \mathrm{on} \quad (0,\infty)\times \mathcal{O}_1.$$
\end{lem}
\begin{proof}
Using the uniqueness of analytic continuation \cite[Corollary 1.2.5, Chapter 1]{krantz2002primer}, since $\widetilde{\sigma}_h(t, \cdot)$ and $\widetilde{\sigma}_p(t, \cdot)$ are analytic continuation (in $t$) of $\sigma_h(t, \cdot)$ and $\sigma_p(t, \cdot)$  on $(0,\infty)$ as functions with value in $L^2(0,\pi)$, 
the condition $\sigma_h=-\sigma_p, \quad \mathrm{on} \quad (0,T)\times \mathcal{O}_1$ gives 
$$ \widetilde{\sigma}_h=-\widetilde{\sigma}_p, \quad \mathrm{on} \quad (0,\infty)\times \mathcal{O}_1.$$
\end{proof}

\vspace{2mm}

\textbf{Proof of \cref{thm_appequivucp} :}
Let $\sigma$ be as stated in \cref{thm_appequivucp} with the splitting \eqref{split_v}. Then due to \cref{lemparabolicpart}, \cref{lemhyperpart} and \cref{lemextwhole}, it follows that $\widetilde{\sigma}\in C([0,\infty); L^2(0,\pi))$ is defined with the same expansion as in \eqref{expansion_v} for all $t\in (0,\infty)$ and  
\begin{equation}\label{eqextcondition}
 \widetilde{\sigma}=0 \quad \mathrm{in} \quad (0,\infty)\times \mathcal{O}_1.
 \end{equation}

Let $m\in\mathbb{N}$ be fixed. Then consider a function $\psi_m\in C^{\infty}(0,\pi)$ having support in $\mathcal{O}_1$ and $\left\langle\cos(m\cdot),\psi_m \right\rangle\neq 0. $
Now, we define
\begin{equation}\label{defn_F_m}
	F_m(t):=\left\langle \widetilde{\sigma}(t,\cdot),\psi_m \right\rangle=\sum\limits_{n=1}^{\infty}\sum\limits_{l=1}^{3}\frac{\alpha_{n,l}}{\psi_{n,l}}e^{\lambda_n^lt}\left\langle\cos(n\cdot),\psi_m \right\rangle + \left\langle\frac{\alpha_0}{\sqrt{b\pi}},\psi_m \right\rangle,\quad t\in (0,\infty).
\end{equation}
Since $\psi_m$ is smooth and the real parts of $\lambda_n^l$ are negative, using \eqref{cgs-theta_nl} and \eqref{est_alphanl}, the series in \eqref{defn_F_m} is uniformly convergent
for all $t\in (0,\infty)$. 

Also since $\psi_m$ has a support in $\mathcal{O}_1$, \eqref{eqextcondition} gives $F_m(t)=0$ for all $t\in (0,\infty)$.

Since $F_m\in L^\infty(0,\infty)$, taking the Laplace transform of $F_m$, we obtain
\begin{equation*}
	\hat{F}_m(\tau)=\int_{0}^{\infty}e^{-t\tau}F_m(t)\rd t,\quad \forall  \tau\in D,
\end{equation*}
where $D=\{\tau\in \mathbb{C}\mid \mathrm{Re}\, \tau>0\}$. 

The condition $F_m=0$ on $(0,\infty)$ gives $\hat{F}_m(\cdot)=0$ on $D$. 

Moreover, using the expansion \eqref{defn_F_m} and the fact that Re$\lambda_n^\ell < 0, \ell=1,2,3, n\in\mathbb{N}$, we obtain
\begin{equation*}
	\hat{F}_m(\tau)=\sum\limits_{n=1}^{\infty}\sum\limits_{\ell=1}^{3}\frac{\alpha_{n,\ell}}{\psi_{n,\ell}(\tau-\lambda_n^\ell)}\left\langle\cos(n\cdot),\psi_m \right\rangle + \frac{1}{\tau}\left\langle\frac{\alpha_0}{\sqrt{b\pi}},\psi_m \right\rangle,\quad\forall   \tau \in D.\:  
\end{equation*}
We can see that the above series converges uniformly on any compact subset of the complex plane not containing $\left\lbrace 0,\lambda_n^\ell, \ell=1,2,3 \right\rbrace_{n\in\mathbb{N}} $ and hence the map $\tau\rightarrow \hat{F}_m(\tau)$ can be extended analytically as a unique function in $\mathbb{C}\backslash\{0, \lambda_n^\ell, \, \ell=1,2,3, \, n\in \mathbb{N}\}$ with poles at $\left\lbrace 0,\lambda_n^\ell, \ell=1,2,3 \right\rbrace_{n\in\mathbb{N}} $ with the representation 
\begin{equation}\label{hatF}
	\hat{F}_m(\tau)=\sum\limits_{n=1}^{\infty}\sum\limits_{\ell=1}^{3}\frac{\alpha_{n,\ell}}{\psi_{n,\ell}(\tau-\lambda_n^\ell)}\left\langle\cos(n\cdot),\psi_m \right\rangle+ \frac{1}{\tau}\left\langle\frac{\alpha_0}{\sqrt{b\pi}},\psi_m \right\rangle,\quad\forall   \tau \in \mathbb{C}\backslash \{0, \lambda_n^\ell, \, \ell=1,2,3, \, n\in \mathbb{N}\}.
\end{equation}
Further, using the fact that $\hat{F}_m=0$ on $D$ and $\hat{F}_m$ is analytic on $\mathbb{C}\backslash \{0, \lambda_n^\ell, \, \ell=1,2,3, \, n\in \mathbb{N}\}$ containing $D$, from \eqref{hatF}, we have 
\begin{equation*}
	0=\sum\limits_{n=1}^{\infty}\sum\limits_{\ell=1}^{3}\frac{\alpha_{n,\ell}}{\psi_{n,\ell}(\tau-\lambda_n^\ell)}\left\langle\cos(n\cdot),\psi_m \right\rangle+ \frac{1}{\tau}\left\langle\frac{\alpha_0}{\sqrt{b\pi}},\psi_m \right\rangle,\quad\forall   \tau \in \mathbb{C}\backslash \{0, \lambda_n^\ell, \, \ell=1,2,3, \, n\in \mathbb{N}\}.
\end{equation*}
Now taking contours $\gamma_m^\ell$ enclosing $\lambda_m^\ell$, for $\ell=1,2,3$ and applying the residue theorem to the above expansion, we obtain
\begin{equation*}
	\frac{\alpha_{m,\ell}}{\psi_{m,\ell}}\left\langle\cos(m\cdot),\psi_m \right\rangle=0,\quad \ell=1,2,3.
\end{equation*}
Since $\left\langle\cos(m\cdot),\psi_m \right\rangle\neq 0$, we get  $\alpha_{m,\ell}=0$ for $\ell=1,2,3$.  Repeating this argument for any $m\in\mathbb{N}\cup\{0\}$, we get 
$\alpha_0=0= \alpha_{m,\ell},\quad \forall\, \ell=1,2,3,\quad \forall\, m\in\mathbb{N}.$
Hence from \eqref{est_alphanl}, we get $\left(\sigma^0,v^0,\tilde S^0 \right)^{\top}=(0,0,0)^{\top}$ in $(0,\pi)$.

Finally, \cref{thm_appequivucp} implies \cref{main_approximate}.

\section{Multiple eigenvalues}\label{secmultiev}

	Recall from (e) of \cref{spectrumofA} that $\mathcal{A}$ can have multiple eigenvalues. Here we will see that the results analogues to \cref{biorthonormality} and \cref{basis_representation} are also hold. Then we can apply our analysis to prove the \cref{thm_pos}.
	
%	We first note:
	\begin{prop}\label{condndoubleroot}
		For each $n\in \mathbb{N}$, let us denote
		\begin{equation}\label{defnq}
			q_n(\lambda)=-b+\frac{  {\lambda}^2 }{\rho_sn^2} - \frac{\mu\kappa{\lambda}^2 }{\rho_s^2\left(  1+\kappa{\lambda}\right)^2}.
		\end{equation}
	Then $\lambda$ is a double root of the characteristic equation \eqref{equ-char} if and only if $q_n(\lambda)=0$.
	\end{prop}
From the relation between roots and coefficients in \eqref{eq2.6}, the above result follows by a direct calculation. 
	
	\vspace{2mm}
	
	\noindent\textbf{Case 1.} {\textit{$\mathcal{A}$ has double eigenvalues :}} Let for some $n\in \mathbb{N}$, $\mathcal{A}$ has eigenvalues $\lambda_n^1, \lambda_n^2, \lambda_n^3$ with $\lambda_n^2=\lambda_n^3 \left( \text{suppose }\lambda_n^2 \text{ is a double eigenvalue}\right) $.
In this case, all three eigenvalues for this $n$, are real. \cref{condndoubleroot} gives that $q_n(\lambda_n^2)=0$.

Now we consider $\xi_{n,1}$, an eigenfunction of $\mathcal{A}$ corresponding to the eigenvalue $\lambda_n^1$, $\xi_{n,2}$, an eigenfunction of $\mathcal{A}$ corresponding to the eigenvalue $\lambda_n^2$ and $\xi_{n,3}$, a generalized eigenfunction of $\mathcal{A}$	corresponding to the eigenvalue $\lambda_n^2$ :
	\begin{align*}
		&\xi_{n,1}=\frac{1}{\theta_{n,1}}\begin{pmatrix}
			-\cos(nx)\vspace{1mm}\\\frac{\lambda_n^1}{{\rho_s}n}\sin(nx)\vspace{1mm}\\\frac{\mu\lambda_n^1}{\rho_s\left( 1+\kappa\lambda_n^1\right) }\cos(nx)
		\end{pmatrix},\quad \xi_{n,2}=\frac{1}{\theta_{n,2}}\begin{pmatrix}
			-\cos(nx)\vspace{1mm}\\\frac{\lambda_n^2}{{\rho_s}n}\sin(nx)\vspace{1mm}\\\frac{\mu\lambda_n^2}{\rho_s\left( 1+\kappa\lambda_n^2\right) }\cos(nx)\end{pmatrix},\\
		& \xi_{n,3}=\frac{1}{\theta_{n,2}} \begin{pmatrix}
			\left( c_n-\frac{1}{\lambda_n^2}\right) \cos(nx)\vspace{1mm}\\-c_n\frac{\lambda_n^2}{\rho_s n}\sin(nx)\vspace{1mm}\\\frac{\mu\kappa\lambda_n^2-c_n\mu\lambda_n^2\left( 1+\kappa\lambda_n^2\right) }{\rho_s\left( 1+\kappa\lambda_n^2\right)^2 }\cos(nx)
		\end{pmatrix},\quad c_n=\frac{b\rho_s^2\left( 1+\kappa\lambda_n^2\right)^4+\mu\kappa^3\left( \lambda_n^2\right)^4  }{\lambda_n^2\left( 1+\kappa\lambda_n^2\right)\left( b\rho_s^2\left( 1+\kappa\lambda_n^2\right)^3+\mu\kappa^2\left( \lambda_n^2\right)^3\right) },
	\end{align*}
	where $\theta_{n,1}$ and $\theta_{n,2}$ are as same as in \eqref{equ-theta_nl}. We can see that $\left(\lambda_n^2\mathcal{I}-\mathcal{A} \right)\xi_{n,3}=\xi_{n,2} $.
	
	Similarly, we consider $\xi^*_{n,1}$, an eigenfunction of $\mathcal{A}^*$ corresponding to the eigenvalue $\lambda_n^1$, $\xi^*_{n,3}$, an eigenfunction of $\mathcal{A}^*$ corresponding to the eigenvalue $\lambda_n^2$ and $\xi^*_{n,2}$, a generalized eigenfunction of $\mathcal{A}^*$	corresponding to the eigenvalue $\lambda_n^2$ :
	
	\begin{align*}
		&\xi^*_{n,1}=\frac{1}{\psi_{n,1}}\begin{pmatrix}
			\cos(nx)\vspace{1mm}\\\frac{\lambda_n^1}{{\rho_s}n}\sin(nx)\vspace{1mm}\\-\frac{\mu\lambda_n^1}{\rho_s\left( 1+\kappa\lambda_n^1\right) }\cos(nx)
		\end{pmatrix},\quad \xi^*_{n,2}=\frac{1}{\psi_{n,2}}\begin{pmatrix}
			\frac{1}{\lambda_n^2}\cos(nx)\vspace{1mm}\\0\vspace{1mm}\\-\frac{\mu\kappa\lambda_n^2}{\rho_s\left( 1+\kappa\lambda_n^2\right)^2 }\cos(nx)
		\end{pmatrix},\\
		& \xi^*_{n,3}=\frac{1}{\psi_{n,2}}\begin{pmatrix}
			\cos(nx)\vspace{1mm}\\\frac{\lambda_n^2}{{\rho_s}n}\sin(nx)\vspace{1mm}\\-\frac{\mu\lambda_n^2}{\rho_s\left( 1+\kappa\lambda_n^2\right) }\cos(nx)\end{pmatrix},
	\end{align*}
	where $\psi_{n,1}$ is as same as in \eqref{equ-psi_nl} and $\psi_{n,2}\neq 0$ is given by 
	\begin{equation}
		\psi_{n,2}=\frac{1}{\theta_{n,2}}\left[ -\frac{\pi}{2}\left(\frac{b}{\lambda_n^2}+\frac{\mu\kappa^2\left( \lambda_n^2\right)^2}{\rho_s^2\left( 1+\kappa\lambda_n^2\right)^3} \right) \right].
	\end{equation}
We can see that $\left(\lambda_n^2\mathcal{I}-\mathcal{A}^* \right)\xi^*_{n,2}=\xi^*_{n,3} $.
	
The families $\left\lbrace \xi_{n,l}\: |\: 1\leq l\leq 3, n\in\mathbb{N}\right\rbrace$ and $\left\lbrace \xi^*_{n,l}\: |\: 1\leq l\leq 3, n\in\mathbb{N}\right\rbrace$ satisfy the bi-orthonormality relations \eqref{biortho} and \cref{basis_representation} holds.

\vspace{2.mm}	

\noindent \textbf{Case 2.} {\textit{$\mathcal{A}$ has triple eigenvalues :}} Let for some $n\in \mathbb{N}$, $\mathcal{A}$ has eigenvalues $\lambda_n^1, \lambda_n^2, \lambda_n^3$ such that $\lambda_n^1= \lambda_n^2=\lambda_n^3=\lambda_n \left(\text{ say}\right)$. 
In this case all eigenvalues for this $n$ are real. Using the relation between roots and coefficients \eqref{eq2.6}, we get
	\begin{align}\label{lambda=3k}
		&\lambda_n=-\frac{1}{3\kappa},\quad\frac{1}{\kappa^2}=27 b \rho_s n^2,\quad\mu=8b\kappa\rho_s^2.
	\end{align}
	We consider $\xi_{n,1}$, an eigenfunction of $\mathcal{A}$ corresponding to the eigenvalue $\lambda_n$, $\xi_{n,2}$, a generalized eigenfunction of $\mathcal{A}$ corresponding to the eigenvalue $\lambda_n$ and $\xi_{n,3}$, a generalized eigenfunction of $\mathcal{A}$	corresponding to the eigenvalue $\lambda_n$ :
	\begin{align*}
		&\xi_{n,1}=\frac{1}{\theta_{n}}\begin{pmatrix}
			-\cos(nx)\vspace{1mm}\\-\frac{1}{3\kappa{\rho_s}n}\sin(nx)\vspace{1mm}\\-\frac{\mu}{2\kappa\rho_s }\cos(nx)
		\end{pmatrix},\quad \xi_{n,2}=\frac{1}{\theta_{n}}\begin{pmatrix}
			\frac{3\kappa}{2}\cos(nx)\vspace{1mm}\\-\frac{1}{2\rho_sn}\sin(nx)\vspace{1mm}\\-\frac{3\mu}{2\rho_s }\cos(nx)\end{pmatrix},\quad \xi_{n,3}=\frac{1}{\theta_{n}}\begin{pmatrix}
			-\frac{27 \kappa^2}{4}\cos(nx)\vspace{1mm}\\-\frac{3\kappa}{4\rho_s n}\sin(nx)\vspace{1mm}\\-\frac{27\mu\kappa}{8\rho_s }\cos(nx)
		\end{pmatrix},
	\end{align*}
	where 
	\begin{equation}
		\theta_n=\sqrt{3\pi b }.
	\end{equation}
	Using \cref{condndoubleroot} and \eqref{lambda=3k}, we see that $\left(\lambda_n\mathcal{I}-\mathcal{A} \right)\xi_{n,2}=\xi_{n,1} $ and $\left(\lambda_n\mathcal{I}-\mathcal{A} \right)\xi_{n,3}=\xi_{n,2} $.
	
	Similarly, we consider $\xi^*_{n,1}$, a generalized eigenfunction of $\mathcal{A}^*$ corresponding to the eigenvalue $\lambda_n$, $\xi^*_{n,2}$, a generalized eigenfunction of $\mathcal{A}^*$ corresponding to the eigenvalue $\lambda_n$ and $\xi^*_{n,3}$, an eigenfunction of $\mathcal{A}^*$	corresponding to the eigenvalue $\lambda_n$ :
	
	\begin{align*}
		&\xi^*_{n,1}=\frac{1}{\psi_{n}}\begin{pmatrix}
			\frac{9\kappa^2}{4}\cos(nx)\vspace{1mm}\\\frac{9\kappa}{4\rho_s n}\sin(nx)\vspace{1mm}\\-\frac{9\mu\kappa}{4\rho_s }\cos(nx)
		\end{pmatrix},\quad \xi^*_{n,2}=\frac{1}{\psi_{n}}\begin{pmatrix}
			-3\kappa\cos(nx)\vspace{1mm}\\0\vspace{1mm}\\\frac{3\mu}{4\rho_s }\cos(nx)\end{pmatrix},\quad \xi^*_{n,3}=\frac{1}{\psi_{n}}\begin{pmatrix}
			\cos(nx)\vspace{1mm}\\-\frac{1}{3\kappa{\rho_s}n}\sin(nx)\vspace{1mm}\\\frac{\mu}{2\kappa\rho_s }\cos(nx)
		\end{pmatrix},
	\end{align*}
	where 
	\begin{equation}
		\psi_n=\frac{1}{\theta_n}\left( -\frac{27b\kappa^2\pi}{4}\right) .
	\end{equation}
	Using \cref{condndoubleroot} and \eqref{lambda=3k}, we see that $\left(\lambda_n\mathcal{I}-\mathcal{A}^* \right)\xi^*_{n,2}=\xi^*_{n,3} $ and $\left(\lambda_n\mathcal{I}-\mathcal{A}^* \right)\xi^*_{n,1}=\xi^*_{n,2} $.
	Here also we can see that the families $\left\lbrace \xi_{n,l}\: |\: 1\leq l\leq 3, n\in\mathbb{N}\right\rbrace$ and $\left\lbrace \xi^*_{n,l}\: |\: 1\leq l\leq 3, n\in\mathbb{N}\right\rbrace$ satisfy the bi-orthonormality relations \eqref{biortho} and \cref{basis_representation} holds.

Since in the case of multiple eigenvalues of $\mathcal{A}$, we obtain the result analogous to \cref{basis_representation} using the generalized eigenvectors of $\mathcal{A}$ and $\mathcal{A}^*$, the proof of the null controllability result and approximate controllability result  can be adapted to the case when $\mathcal{A}$ has multiple eigenvalues. The essential  modifications of the proof of the controllability results to this case are indicated below. 

\begin{rem}
In the Case 1, $\mathcal{A}_n$ in \cref{lem-An} and $\mathcal{B}_n$ in \cref{lem-Bn} are given by 
	\begin{equation}
		\mathcal{A}_n=\begin{pmatrix}
			\lambda_n^1&0&0\\
			0&\lambda_n^2&-1\\
			0&0&\lambda_n^2
		\end{pmatrix}, \quad \mathcal{B}_n= b\sqrt{\frac{\pi}{2}}\begin{pmatrix}
		\frac{1}{{\psi_{n,1}}}\\\frac{1}{{\lambda_n^2\psi_{n,2}}}\\\frac{1}{{\psi_{n,2}}}
	\end{pmatrix}.
	\end{equation}

In the Case 2, $\mathcal{A}_n$ in \cref{lem-An} and $\mathcal{B}_n$ in \cref{lem-Bn} are given by 
	\begin{equation}
		\mathcal{A}_n=\begin{pmatrix}
			\lambda_n&-1&0\\
			0&\lambda_n&-1\\
			0&0&\lambda_n
		\end{pmatrix},\quad \mathcal{B}_n= b\sqrt{\frac{\pi}{2}}\begin{pmatrix}
			\frac{9\kappa^2}{{4\psi_{n}}}\vspace{1mm}\\\frac{-3\kappa}{{\psi_{n}}}\vspace{1mm}\\\frac{1}{{\psi_{n}}}
		\end{pmatrix}.
	\end{equation}

We can see that in the both Cases 1 and 2, 
\begin{align*}
		\text{Rank}\left[\lambda I-\mathcal{A}_n \;;\; \mathcal{B}_n \right]=3, \: \text{ for }\lambda=\lambda^1_n, \lambda^2_n, \lambda^3_n. 
	\end{align*}
Thus the finite dimensional system  \eqref{opp-eqn_proj} is controllable. 

The rest of the proof of \cref{thm_pos} in the case of multiple eigenvalues is same as given in Section \ref{secnull}. 
\end{rem}

\begin{rem}\label{remmultipropunq}
In the Case 1, to prove Theorem \ref{thm_appequivucp}, let $m\in \mathbb{N}$ such that $\mathcal{A}$ has eigenvalues $\lambda_m^1, \lambda_m^2, \lambda_m^3$ with $\lambda_m^2=\lambda_m^3$. We note that 
$$ e^{tA^*}\xi^*_{m,2}= e^{t\lambda^2_m}\xi^*_{m,2}- te^{t\lambda^2_m}\xi^*_{m,3}.$$
Using the above expression, $F_m$ in \eqref{defn_F_m} can be modified accordingly and its Laplace transform for all  $\tau \in \mathbb{C}\backslash \{0, \lambda_n^\ell, \, \ell=1,2,3, \, n\in \mathbb{N}\}$ is 
\begin{equation*}
\hat{F}_m(\tau)= \Big[\frac{\alpha_{m,3}}{\psi_{m,3}(\tau-\lambda_m^2)}-\frac{\alpha_{m,2}}{\psi_{m,3}(\tau-\lambda_m^2)^2}+\frac{\alpha_{m,2}}{\psi_{m,2}\lambda^2_m(\tau-\lambda_m^2)}\Big]\left\langle\cos(m\cdot),\psi_m \right\rangle+ g(\tau),
\end{equation*}
where $g(\cdot)$ is analytic at $\tau=\lambda^2_m$.

 Further, as obtained in the proof of Theorem \cref{thm_appequivucp} in Section \ref{Approximate controllability}, we have 
$ \hat{F}_m(\tau)=0$ for all  $\tau \in \mathbb{C}\backslash \{0, \lambda_n^\ell, \, \ell=1,2,3, \, n\in \mathbb{N}\}$. 
Multiplying the above relation with $(\tau-\lambda^2_m)$ and taking contour around $\lambda^2_m$ and applying  Cauchy's Integral formula to $(\tau-\lambda^2_m)\hat{F}_m(\tau)$, we get that $ \alpha_{m,3}=0$.  Taking contour again around $\lambda^2_m$ and applying Cauchy's Integral formula to $\hat{F}_m(\tau)$, we obtain $\alpha_{m,2}=0$. 

In the Case 2, an analogous argument discussed above is applicable. 

The rest of the proof of \cref{thm_appequivucp} in the case of multiple eigenvalues is similar to that of \cref{thm_appequivucp} done in Section \ref{Approximate controllability}.

\end{rem}

\section{Appendix}\label{secappen}

Let us recall that $\{\mu_n\}_{|n|\ge N_0}$ defined in \eqref{defnmun}, a sequence of complex number satisfying \\
Hypothesis $\mathcal{H}_1$:
\begin{enumerate}
\item For all $|n|\ge N_0$, $\mathrm{Re}\, \mu_n=n\sqrt{(b\rho_s+\frac{\mu}{\kappa \rho_s})}+\delta_n$, and $\mathrm{Im} \, \mu_n=\frac{1}{2}(\omega_0+\frac{1}{\kappa})-\epsilon_n$ and 
$$ \epsilon_n\rightarrow 0, \quad \delta_n\rightarrow 0, \quad |n|\rightarrow \infty, \quad |\delta_n|\le \frac{1}{2}\sqrt{(b\rho_s+\frac{\mu}{\kappa \rho_s})}, \quad \forall\, |n|\ge N_0.$$

\item $\mathrm{Re}\,\mu_{n+1}-\mathrm{Re}\,\mu_n\ge \gamma>0, \quad \forall\, |n|\ge N_0.$
\end{enumerate}

We have the following Ingham inequality for complex frequencies with convergent real part. The proof is a minor modification of Ingham's original argument 
\cite{ingham1936some} in the case when $\mathrm{Im}\;\mu_n=0$. A version of this inequality was used in \cite{Renardy05} and \cite{chowdhury2014null}.
\begin{prop}\label{propI-4}
Let $N_0$ and $\{\mu_n\}_{|n|\ge N_0}$ satisfy Hypothesis $\mathcal{H}_1$. Let $ T> \frac{2\pi}{\gamma}$.
There exist $M\ge N_0$ and positive constants $C$ and $C_1$ depending on $T$ such that for
$ g(t)= \displaystyle\sum_{|n|\ge M}{\beta_n e^{i\mu_n t}}$ with
$\displaystyle{\sum_{|n|\ge M}}|\beta_n|^2<\infty$, the following inequality holds:
\begin{equation}\label{I-4}
 C\sum_{|n|\ge M}|\beta_n|^2\leq \int_0^T|g(t)|^2dt\leq C_1\sum_{|n|\ge M}|\beta_n|^2.
\end{equation}
\end{prop}
\begin{proof}
We shall prove the left side inequality of \eqref{I-4} in two steps. In the first step we assume $(\omega_0+\frac{1}{\kappa})=0$ and in the next step we consider
 $(\omega_0+\frac{1}{\kappa})$ as a nonzero real number.

 Step 1. Let $T>\frac{2\pi}{\gamma}$. Our claim is that there exist $M\ge N_0$ and a positive $C$ such that
 $$C\sum_{|n|\ge M}|\beta_n|^2\leq \int_0^T|g(t)|^2dt,$$
 for $ g(t)= \sum_{|n|\ge M}{\beta_n e^{i\mu_nt}}$ with $\displaystyle{\sum_{|n|\ge M}}|\beta_n|^2<\infty$, where $ \mu_n= n\sqrt{(b\rho_s+\frac{\mu}{\kappa \rho_s})}+\delta_n-i\epsilon_n$. 
 
 Notice that we have $$ Re\,\mu_{n+1}-Re\,\mu_{n}\ge \gamma>0, \quad \forall \; |n|\ge N_0, \quad \mbox{ and}\quad T>\frac{2\pi}{\gamma}.$$
 We reduce the problem to the case $ T= 2\pi$ and $\gamma>1$.
 %As we have $ T $ and $\vartheta$ satisfying $ T\vartheta>2\pi$,
Denoting $\theta_n=\frac{T}{2\pi}\mu_n $, it follows that $ Re\, \theta_{n+1}-Re\, \theta_{n}=\frac{T\gamma}{2\pi} > 1 $.
 Then
 $$ \int_0^T|g(t)|^2dt=\int_0^T |\sum_{|n|\ge M}{\beta_n e^{i\mu_nt}}|^2dt= \frac{T}{2\pi}
 \int_0^{2\pi}|\sum_{|n|\ge M}{\beta_n e^{i\theta_n s}}|^2\,ds.$$

Now we prove that there exists $C'_1>0$ such that
\begin{equation}\label{eqI1-4} C'_1\sum_{|n|\ge M}|\beta_n|^2\leq \int_{-\pi}^\pi|\sum_{|n|\ge M}{\beta_n e^{i\theta_nt}}|^2dt.\end{equation}
Define  the function
$$ h:\R\rightarrow \R, \quad h(t)= \left\{\begin{array}{ll} \cos(\frac{t}{2}), & |t|<\pi, \\ 0, & |t|>\pi. \end{array}\right.$$
Its Fourier transform $K(\phi)$ is
$$ K(\phi)= \int_{-\pi}^\pi h(t)e^{it\phi}dt= \frac{4\cos(\pi \phi)}{1-4\phi^2}. $$
Since $0\leq h(t)\leq 1$ for any $t\in [-\pi,\pi],$  we have that
\begin{eqnarray*}
 \lefteqn{\int_{-\pi}^\pi|\sum_{|n|\ge M}{\beta_n e^{i\theta_nt}}|^2dt\geq \int_{-\pi}^{\pi}h(t)|\sum_{|n|>N}{\beta_n e^{i\theta_nt}}|^2dt}\\
 & & \hspace{3cm} = \int_{-\pi}^{\pi}h(t)\sum_{|k|\ge M}\sum_{|n|\ge M}\beta_k\overline\beta_ne^{i(\theta_k-\overline{\theta_n})t}dt\\
 & & \hspace{3cm} = \sum_{|k|\ge M}\sum_{|n|\ge M}\beta_k\overline\beta_n K(\theta_k-\overline{\theta_n})\\
 & & \hspace{3cm} = \sum_{|k|\ge M}|\beta_k|^2K(\theta_k-\overline{\theta_k})+\sum_{|k|\ge M}\sum_{|n|\ge M, n\neq k}\beta_k\overline{\beta_n}K(\theta_k-\overline{\theta_n})\\
 & & \hspace{3cm} = I+II.
\end{eqnarray*}
Our aim is to show $ |II|<(1-\eta)I$, for some $\eta\in (0,1)$. Note that $|K(\theta_k-\overline{\theta_k})|= K(\theta_k-\overline{\theta_k})$ and $ |K(\theta_k-\overline{\theta_n})|= |K(\theta_n-\overline{\theta_k})|$.
Then by the Cauchy-Schwartz inequality we have
\begin{eqnarray*} \lefteqn{|II|= |\sum_{|k|\ge M}\sum_{|n|\ge M, n\neq k}\beta_k\overline{\beta_n}K(\theta_k-\overline{\theta_n})|}\\
& & \leq \sum_{|k|\ge M}\sum_{|n|\ge M, n\neq k}\frac{|\beta_k|^2+|\beta_n|^2}{2}
|K(\theta_k-\overline{\theta_n})|=\sum_{|k|\ge M}|\beta_k|^2\sum_{|n|\ge M, n\neq k}|K(\theta_k-\overline{\theta_n})|.
\end{eqnarray*}
Therefore finally we have to show that for some $\eta\in (0,1)$
\begin{equation}\label{I1-4}
\sum_{|n|\ge M, n\neq k}|K(\theta_k-\overline{\theta_n})|\leq (1-\eta)K(\theta_k-\overline{\theta_k}), \quad \forall |k|\ge M.
\end{equation}
By calculation, since $|\epsilon_n|\rightarrow 0$ and $|\delta_n|\rightarrow 0$ as $|n|\rightarrow\infty$, choosing large enough $M\geq N_0$, we get that
\begin{equation}\label{I2-4}
|K(\theta_k-\overline{\theta_n})|\leq \frac{4\cosh(\frac{T}{2}(\epsilon_k+\epsilon_n))}{4(k-n)^2\vartheta^2-1}, \quad \forall\, |n|\ge M, \, |k|\ge M, \, k\neq n, 
\end{equation}
for some $\vartheta$ such that $1<\vartheta<\frac{T\gamma}{2\pi}$, 
 and \begin{equation}\label{I3-4}
      K(\theta_k-\overline{\theta_k})=\frac{4\cosh(T\epsilon_k)}{1+\frac{4T^2}{\pi^2}\epsilon_k^2}, \quad \forall\, |k|\ge M.
      \end{equation}
Since $ \epsilon_n\rightarrow 0$ as $|n|\rightarrow\infty$, for a given $\delta>0$, by choosing $M\geq N_0$ suitably further large we get $|\epsilon_n|<\frac{2\pi}{T}\delta$ for all $|n|\ge M$. Then  \eqref{I2-4} and \eqref{I3-4}
gives
$$ \sum_{|n|\ge M,n\neq k}|K(\theta_k-\overline{\theta_n})|<\frac{4}{\vartheta^2}\cosh(2\pi\delta);\quad K(\theta_k-\overline{\theta_k})\geq\frac{4}{1+16\delta^2}.$$
As $ \vartheta>1$ we can choose $\delta$ sufficiently small such that for some $\eta \in (0,1)$
$$ \frac{4}{\vartheta^2}\cosh(2\pi\delta)< \frac{4(1-\eta)}{1+16\delta^2},$$
which implies \eqref{I1-4}. Thus \eqref{eqI1-4} is proved. Since we have
$$ \int_0^{2\pi}|\sum_{|n|\ge M}{\beta_n e^{i\theta_ns}}|^2ds=\int_{-\pi}^{\pi}|\sum_{|n|\ge M}{\beta_n e^{i\theta_n(t+\pi)}}|^2dt,$$
the inequality \eqref{eqI1-4} and  the existence of a positive lower bound of $|e^{i\theta_n\pi}|^2$, for all $|n|\ge M$
(due to the converging real part of $\; i\theta_n\pi\;$) gives
$$ C\sum_{|n|\ge M}|\beta_n|^2\leq C'_1\sum_{|n|\ge M}|\beta_n|^2|e^{i\theta_n\pi}|^2\leq
\int_{-\pi}^\pi|\sum_{|n|\ge M}{\beta_n e^{i\theta_n(t+\pi)}}|^2dt,$$
and hence
\begin{equation}\label{I4-4}
 C\sum_{|n|\ge M}|\beta_n|^2\leq \int_0^T|g(t)|^2dt.
\end{equation}

Step 2. Now consider $ g(t)= \sum_{|n|\ge M}{\beta_n e^{i\mu_nt}}$ with $ i\mu_n= -(\omega_0+\frac{1}{\kappa})+\varepsilon_n+i\Big(n\sqrt{(b\rho_s+\frac{\mu}{\kappa\rho_s})}+\delta_n\Big)$ and $(\omega_0+\frac{1}{\kappa})\neq 0 $.
Then  we have
\begin{equation}\label{I5-4}
 \int_0^T|e^{(\omega_0+\frac{1}{\kappa})t}g(t)|^2dt= \int_0^T(e^{2(\omega_0+\frac{1}{\kappa})t})|g(t))|^2dt\leq c\int_0^T|g(t)|^2dt.
\end{equation}
Now we  have inequality \eqref{I4-4} for $\int_0^T|e^{(\omega_0+\frac{1}{\kappa})t}g(t)|^2dt.$  From that we conclude the left inequality of \eqref{I-4}.

The right hand side inequality of \eqref{I-4} can be proved in the same fashion as in the standard case of Ingham's
inequality. For details, one may see the proof of the
Ingham inequality in
\cite{ingham1936some} or in the lecture note \cite{micu2004introduction} (Chapter 4, Theorem 4.2).
\end{proof}

\bibliographystyle{siam}
\bibliography{Maxwell_zero}{}
	
\end{document}